\documentclass[11pt,reqno]{amsart}
\usepackage[a4paper]{anysize}\marginsize{2cm}{2cm}{2cm}{2cm}
\pdfpagewidth=\paperwidth \pdfpageheight=\paperheight
\allowdisplaybreaks
\usepackage[english]{babel}
\usepackage[latin1]{inputenc}
\usepackage[T1]{fontenc}
\usepackage{amssymb,amsmath,amstext,amsfonts,amsthm}
\usepackage{mathtools}
\usepackage{verbatim}
\usepackage{hhline}
\usepackage{relsize}
\usepackage[colorlinks=true, urlcolor=blue, linkcolor=blue, citecolor=blue]{hyperref}
\usepackage{fancyvrb}

\numberwithin{equation}{section}
\theoremstyle{plain}
\newtheorem*{theorem*}{Theorem}
\newtheorem{theorem}{Theorem}
\numberwithin{theorem}{section}
\newtheorem{proposition}[theorem]{Proposition}
\newtheorem{lemma}[theorem]{Lemma}
\newtheorem{corollary}[theorem]{Corollary}
\newtheorem{conjecture}[theorem]{Conjecture}

\theoremstyle{definition}
\newtheorem{definition}[theorem]{Definition}
\newtheorem{remark}[theorem]{Remark}
\newtheorem{example}[theorem]{Example}

\newcommand{\C}{\mathbb{C}}

\newcommand{\Z}{\mathbb{Z}}
\newcommand{\PP}{\mathbb{P}}

\newcommand{\R}{\mathbb{R}}

\newcommand{\mR}{\mathsmaller{\mathbb{R}}}
\newcommand{\mC}{\mathsmaller{\mathbb{C}}}

\newcommand{\blue}[1]{{\color{blue}#1}}

\makeatletter
\newcommand*{\rom}[1]{\expandafter\@slowromancap\romannumeral #1@}
\makeatother

\makeatletter
\@namedef{subjclassname@2020}{%
	\textup{2020} Mathematics Subject Classification}
\makeatother

\date{}
\begin{document}

\author{Giorgio Ottaviani}
\address{Dipartimento di Matematica e Informatica ``Ulisse Dini'', Universit\`a di Firenze, Italy}
\email{giorgio.ottaviani@unifi.it}

\author{Luca Sodomaco}
\address{Department of Mathematics and Systems Analysis, Aalto University, Espoo, Finland}
\email{luca.sodomaco@aalto.fi}

\author{Emanuele Ventura}
\address{Mathematisches Institut, Universit\"{a}t Bern, Switzerland}
\email{emanueleventura.sw@gmail.com; emanuele.ventura@math.unibe.ch}


\title[Asymptotics of degrees and ED degrees of Segre products]{Asymptotics of degrees and ED degrees of Segre products}

\begin{abstract}
Two fundamental invariants attached to a projective variety are its classical algebraic degree and its Euclidean Distance degree (ED degree). 
In this paper, we study the asymptotic behavior of these two degrees of some Segre products and their dual varieties. We analyze the asymptotics of degrees of (hypercubical) hyperdeterminants, the dual hypersurfaces to Segre varieties. 

We offer an alternative viewpoint on the stabilization of the ED degree of some Segre varieties. Although this phenomenon was incidentally known from Friedland-Ottaviani's formula expressing the number of singular vector tuples of a general tensor, our approach provides a geometric explanation. 

Finally, we establish the stabilization of the degree of the dual variety of a Segre product $X\times Q_{n}$, where $X$ is a projective variety and $Q_n\subset \PP^{n+1}$ is a smooth quadric hypersurface. 
\end{abstract}

\maketitle

\section{Introduction}

Let $V^\mR$ be a real vector space equipped with a distance function and let $X\subset \PP(V^\mR\otimes_\mR\C)$ be a complex projective variety. Two fundamental features of $X$ are its degree and its {\it Euclidean Distance degree} (ED degree). While the first is one of the basic numerical invariants of an algebraic variety, the second was recently introduced in \cite{DHOST} and since then has found several interesting applications in Pure and Applied Algebraic Geometry \cite{DH, QCL, DLOT, DOT, BKL, HL, HW}.

The aim of this paper is to initiate a study of the asymptotic behavior of these two important notions attached to some special varieties and their duals. These varieties are {\it Segre products}, i.e. images of direct products of projective varieties through the Segre embedding. 
Our perspective is naturally inspired by the recently emerging interest in {\it stabilization} properties in Algebraic Geometry and Representation Theory, seeking for results about large families of related varieties at once, rather than specific instances. The discussion around \cite[Conjecture 1.3]{DH} motivated this work. Indeed we found the geometrical explanation sought in \cite{DH} (see Corollary \ref{cor: stabED} and its proof). We hope that the specialization technique of this paper may open the road towards new results, including the above conjecture.

Whenever a classical Segre variety $\PP^{k_1}\times \cdots \times \PP^{k_d}\subset \PP(\C^{k_1+1}\otimes \cdots \otimes \C^{k_d+1})$ is not {\it dual defective}, i.e. its dual variety is a hypersurface, the polynomial defining the latter is the {\it hyperdeterminant} \cite[Chapter 14]{GKZ}.
These higher analogues of matrix determinants are of utmost importance and yet their properties are far from being completely understood.
Hyperdeterminants play a prominent role behind all the results in this paper. 

We now showcase our results. We start off from the very classical (hypercubical) hyperdeterminants of the Segre variety $\mathbb P(\mathbb C^{n+1})^{\times d}$ establishing an asymptotic formula of its degree as the dimension $n+1$ of the vector space goes to infinity.

\begin{theorem*}[{\bf Theorem \ref{thm: hyperdetapprox}}]
Let $N(n{\bf 1}^d)$ be the degree of the hyperdeterminant of format $(n+1)^{\times d}$. Then asymptotically, 
for $d\ge 3$ and $n\to +\infty$, 
\[
N(n{\bf 1}^d) \approx \frac{(d-1)^{2d-2}}{[2\pi(d-2)]^{\frac{d-1}{2}}\,d^{\frac{3d-6}{2}}}\cdot \frac{(d-1)^{dn}}{n^{\frac{d-3}{2}}}\,. 
\]
\end{theorem*}

This should be compared to the behavior of the ED degree (with respect to the Frobenius inner product introduced in Definition \ref{def: Frobenius}) of the Segre variety $X = \mathbb P(\mathbb C^{n+1})^{\times d}$ for $d\ge 3$ and $n\to +\infty$, found by Pantone \cite{Pan}:
\[
\mathrm{EDdegree}_F(X) \approx \frac{(d-1)^{d-1}}{(2\pi)^{\frac{d-1}{2}}\,(d-2)^{\frac{3d-1}{2}}\,d^{\frac{d-2}{2}}}\cdot \frac{(d-1)^{dn}}{n^{\frac{d-1}{2}}}\,. 
\]

The result is shown using a theorem of Raichev-Wilson \cite{RW} and performing a similar analysis to the one of Pantone \cite{Pan}, for approximating the number of singular vector tuples introduced by Lim \cite{L} and Qi \cite{Q}.
Indeed, the source of this result is the surprising similarity between two infinite series a priori unrelated: the generating function of the degrees of hyperdeterminants \cite[Theorem 2.4, Chapter 14]{GKZ} and the generating function of Ekhad and Zeilberger \cite[Theorem 1.2]{EZ} resulting from the ED degree formula of Theorem \ref{thm: FO formula}.
Note that the hyperdeterminant grows faster, unless the trivial case $d=2$ of ordinary matrices, where both sequences collapse to $n+1$ (here the hyperdeterminant coincides with the determinant of a square matrix of order $n+1$, whereas $\mathrm{EDdegree}_F(X) = n+1$ by the Eckart-Young Theorem).
This is in contrast with the first nontrivial case of $2\times 2\times 2$ tensors, where the degree of the hyperderminant is $4$ and $\mathrm{EDdegree}_F(X) = 6$.

Another asymptotic result is proved for $n=1$, i.e., when $X$ is a Segre product of $d$ projective lines, the basic space of qubits in Quantum Information Theory.
In this case, we let $d\to +\infty$ (see Proposition \ref{prop: binaryasympt}).

The stabilization of the Frobenius ED degree of some Segre varieties is a very interesting phenomenon.
This is apparent from expanding Friedland-Ottaviani's formula, which expresses the number of singular vector tuples of a general tensor in a given format.
However, to the best of our knowledge, a geometrical explanation was missing.
We fill this gap providing a geometric view of this behavior that is tightly related to hyperdeterminants.

\begin{theorem*}[{\bf Theorem \ref{thm: specialization} and Corollary \ref{cor: stabED}}]
Let $\dim(V_i) = n_i + 1$ and $\dim(W) = m+1$, $N = \sum_{i=1}^{d}n_i$ and $m\ge N$.
Let $\mathrm{Det}$ be the hyperdeterminant in the boundary format $(n_1+1)\times \cdots \times(n_d+1)\times(N+1)$. Consider a tensor $t\in \PP(V_1\otimes\cdots\otimes V_{d}\otimes\C^{N+1})\subset\PP(V_1\otimes\cdots\otimes V_{d}\otimes W)$ with $\mathrm{Det}(t)\neq 0$. Then the critical points of $t$ in $\PP(V_1)\times\cdots\times\PP(V_{d})\times\PP(W)$ lie in the subvariety $\PP(V_1)\times\cdots\times\PP(V_{d})\times\PP(\C^{N+1})$.

\noindent Moreover, for all $m\geq N$, we have
\[
\mathrm{EDdegree}_F\left(\PP^{n_1}\times\cdots\times\PP^{n_{d}}\times\PP^{m}\right) = \mathrm{EDdegree}_F\left(\PP^{n_1}\times\cdots\times\PP^{n_{d}}\times\PP^{N}\right)\,.
\]
\end{theorem*}
Furthermore, Theorem \ref{thm: descriptionofcriticalpoints} gives a rather detailed description of the critical points of a tensor $t$ in the hypothesis of the previous result: they exhibit an interesting behavior as they distribute either on the subspace where the tensor $t$ lives or on its orthogonal.

Our last contribution deals with the stabilization property of the degree of the dual varieties of some particular Segre products of projective varieties.

\begin{theorem*}[{\bf Corollary \ref{cor: stabilization XxQ}}]
Let $X$ be a projective variety of dimension $m$. For $n\ge 0$, let $Q_n\subset\PP^{n+1}$ be a smooth quadric hypersurface. 
Suppose that $(X\times Q_m)^\vee$ is a hypersurface. Then $(X\times Q_n)^\vee$ is a hypersurface of the same degree as $(X\times Q_m)^\vee$ for all $n\ge m$.
\end{theorem*}

At first glance, the two stabilization phenomena described in Corollaries \ref{cor: stabED} and \ref{cor: stabilization XxQ} seem to be independent. Surprisingly, a relationship appears when $X$ is a Segre product $\PP^{k_1}\times\cdots\times\PP^{k_\ell}$.
In this case the variety $X\times Q_n$ is the largest irreducible component of isotropic elements in $Y=X\times \PP^{n+1}$
and it is, for large $n$, the only irreducible component such that its dual is a hypersurface. If we could apply \cite[Theorem 6.3]{OS} which states $2\mathrm{EDdegree}(Y)=\deg(Y\cap Q)^\vee$, then the equivalence between Corollaries \ref{cor: stabED} and \ref{cor: stabilization XxQ} would follow. Although \cite[Theorem 6.3]{OS} cannot be directly applied in this case, this was a guide for our study. Indeed, from the study of the {\it ED polynomial} of the Segre variety $X\times\PP^{n+1}$ (with respect to the Frobenius inner product), we observed in various examples that $\mathrm{EDdegree}_F(X\times\PP^{n+1})$ stabilizes for $n\to+\infty$ if $\deg[(X\times Q_n)^\vee]$ stabilizes too, and vice-versa (see \S\ref{sec: edpolysegre} for some related open problems). In a second step, we generalized Corollary \ref{cor: stabilization XxQ} to an arbitrary projective variety $X$ using an explicit description of the {\it polar classes} \cite[\S 3]{Hol} of $X\times Q_n$, showing several binomial identities involving them.

The paper is organized as follows. In \S\ref{sec: prelim}, we fix notation and introduce the terminology used throughout the article. 
In \S\ref{sec: hyper}, we derive the asymptotic formula for the degree of the (hypercubical) hyperdeterminants.  As an additional result, in \S\ref{sec: anotherhyperdet} we give an asymptotic formula for the degree of the hyperdeterminant
of format $2^{\times d}$ as $d$ goes to infinity. 
In \S\ref{sec: stabilization}, we first provide a description of critical points for tensors of boundary format as in Theorem \ref{thm: descriptionofcriticalpoints}. Thereafter we show Theorem \ref{thm: specialization} and Corollary \ref{cor: stabED}, along with recording other observations. In \S\ref{sec: stabilization degree dual of XxQ}, we recall polar classes and utilize them to establish the stabilization of the degree featured in Corollary \ref{cor: stabilization XxQ}. In \S\ref{sec: edpolysegre}, we mention some by-products of the stabilization above and formulate some intriguing conjectures, naturally emerging from our experiments.

\section{Preliminaries}\label{sec: prelim}

Throughout the paper, we let $V_1,\ldots,V_d$ be complex vector spaces of dimensions $\dim(V_i)=n_i+1$, respectively.
A (complex) tensor of {\it format} $(n_1+1)\times \cdots\times(n_d+1)$ is a multilinear map $t\colon V_1^*\times \cdots \times V_d^*\rightarrow \C$, i.e., an element of the tensor product (over $\mathbb C$) $V\coloneqq V_1\otimes \cdots \otimes V_d$. 

\begin{definition}\label{def: Segre embedding}
A tensor $t\in V$ is of {\it rank one} (or {\it decomposable}) if $t = v_1\otimes\cdots\otimes v_d$ for some vectors $v_j\in V_j$ for all $j\in[d]\coloneqq\{1,\ldots,d\}$. Tensors of rank at most one in $V$ form the affine cone over the {\it Segre variety of format $(n_1+1)\times \cdots \times (n_d+1)$}, that is the image of the projective morphism
\[
\mathrm{Seg}\colon \PP(V_1)\times \cdots \times \PP(V_d)\to\PP(V)
\]
defined by $\mathrm{Seg}([v_1],\ldots,[v_d])\coloneqq[v_1\otimes\cdots\otimes v_d]$ for all non-zero $v_j\in V_j$. The Segre variety introduced above is often denoted simply by $\PP(V_1)\times \cdots \times \PP(V_d)$.
\end{definition}

For the ease of notation, in the rest we will abuse notation identifying a tensor $t\in V$ with its class in projective space.

\begin{definition}[{\bf Dual varieties}]\label{def: dual varieties}
Let $X\subset \PP(W)$ be a projective variety, where $\dim(W)=n+1$.
Its {\it dual variety} $X^{\vee}\subset\PP(W^*)$ is the closure of all hyperplanes tangent to $X$ at some smooth point \cite[Chapter 1]{GKZ}.
The {\it dual defect} of $X$ is the natural number $\delta_X \coloneqq n-1-\dim(X^{\vee})$. A variety $X$ is said to be {\it dual defective} if $\delta_X>0$. Otherwise, it is {\it dual non-defective}. When $X = \PP(W)$, taken with its tautological embedding into itself, $X^{\vee} = \emptyset$ and $\mathrm{codim}(X^{\vee}) = n+1$. 
\end{definition}

Of particular interest are dual varieties of Segre varieties, whose non-defectiveness is characterized by the following result.

\begin{theorem}{\cite[Chapter 14, Theorem 1.3]{GKZ}}
Let $X = \PP(V_1)\times \cdots \times \PP(V_d)$ be the Segre variety of format $(n_1+1)\times \cdots \times (n_d+1)$, and let $n_j = \max\lbrace n_i\mid i\in[d]\rbrace$. Then $X$ is dual non-defective if and only if
\[
n_j\leq \sum_{i\neq j} n_i\quad\forall\,j\in[d]\,.
\]
\end{theorem}

\begin{definition}[{\bf Hyperdeterminants}]\label{def: hyperdeterminants}
Let $X = \PP(V_1)\times \cdots \times \PP(V_d)$ be the Segre variety of format $(n_1+1)\times \cdots \times(n_d+1)$. When $X$ is dual non-defective, the polynomial equation defining the hypersurface $X^\vee\subset\PP(V^*)$ (up to scalar multiples) is called the {\it hyperdeterminant} of format $(n_1+1)\times \cdots \times (n_d+1)$ and is denoted by $\mathrm{Det}$. When the format $(n_1+1)\times\cdots\times (n_d+1)$ is such that $n_j= \sum_{i\neq j} n_i$ for $n_j = \max \lbrace n_i\mid i\in[d]\rbrace$, the hyperdeterminant is said to be of {\it boundary format}. The hyperdeterminant of format $(n+1)^{\times d}$ is said to be {\it hypercubical}. 
\end{definition}

\section{Asymptotics of degrees of some hyperdeterminants}\label{sec: hyper}

In this section, we establish the asymptotic results for hyperdeterminants of formats $(n+1)^d$ (as $n\rightarrow +\infty$) and
formats $2^d$ (as $d\rightarrow +\infty$).

\subsection{Asymptotics for the hyperdeterminant of \texorpdfstring{$(\mathbb P^n)^{\times d}$}{Pn d}}

Let ${\bf x}=(x_1,\ldots, x_d)$ be a coordinate system in $\C^d$. Let $\alpha=(\alpha_1,\ldots, \alpha_d)\in \mathbb N^d$ be a $d$-tuple of natural numbers. Define ${\bf x}^{\alpha} \coloneqq x_1^{\alpha_1}\cdots x_d^{\alpha_d}$. 

Raichev and Wilson \cite{RW} gave a method to find the asymptotic behavior of the coefficients of a multivariate power series $\sum_{\alpha\in \mathbb N^d} f_{\alpha} {\bf x}^{\alpha}$ in the variables $x_1,\ldots,x_d$, which is the Taylor expansion of a function $F = G/H^{p}$, where $G$ and $H$ are holomorphic functions in a neighborhood of the origin of $\C^d$. They showed that the asymptotics of $f_{n\alpha}$, for $\alpha\in\mathbb N^d$, is governed by special smooth points of the complex (analytic) variety $\mathcal{V}\coloneqq\left\lbrace H ({\bf x}) = 0\right\rbrace\subset \C^d$. 

\begin{definition}[{\bf Strictly minimal point}]\label{def: strictly minimal point}
Let $\mathcal V\subset\C^n$ be a complex variety.
For a point ${\bf c}\in \mathbb C^d$, its {\it polydisc} is $\mathbb D({\bf c}) \coloneqq \lbrace {\bf y}\in \C^d \mbox{ such that } |y_i| \leq |c_i| \mbox{ for all } i \rbrace$.
A point ${\bf c}\in \mathcal V$ is {\it strictly minimal} if it is the only point in $\mathcal V \cap \mathbb D({\bf c})$. 
\end{definition}

\begin{definition}[{\bf Critical point}]\label{def: critical point}
Let $\mathcal V\subset\C^n$ be the variety defined by $H({\bf x}) = 0$.
A smooth point ${\bf c}=(c_1,\ldots, c_d)\in \mathcal V$ is {\it critical} if it is smooth and $c_1\partial_1H({\bf c}) = c_2\partial_2H({\bf c}) = \cdots = c_d\partial_d H({\bf c})$.
It is {\it isolated} if there is a neighborhood of ${\bf c}\in \mathcal V$ where it is the only critical point.
\end{definition}

\begin{theorem}[{\bf \cite[Theorem 3.2]{RW}}]\label{thm: asympt}
Let $d\geq 2$ and let $F = G/H^2$, whose Taylor expansion in a neighborhood of the
origin is $\sum_{\alpha\in \mathbb N^d} f_{\alpha} {\bf x}^{\alpha}$. Suppose ${\bf c}\in \mathcal V = \lbrace H({\bf x}) = 0 \rbrace$ is a smooth with $c_d\partial_d H({\bf c}) \neq 0$, strictly minimal, critical, isolated, and non-degenerate point. Then, for all $N\in\mathbb N$, as $n\rightarrow+\infty$, 
\[
f_{n {\bf 1}} \approx {\bf c}^{-n{\bf 1}} \left[ \left((2\pi n)^{d-1} \det \tilde{g}''(0) \right)^{-1/2}\sum_{j=0}^{1}\sum_{k<N} \frac{(n+1)^{2-1-j}}{(2-1-j)!\,j!}n^{-k} L_k(\tilde{u}_j, \tilde{g}) + O\left(n^{1-(d-1)/2-N}\right)\right]\,.
\]
\noindent In the original formula of Raichev-Wilson, we substituted $\alpha = (1,\ldots,1) = {\bf 1}$ and $p=2$. 
\end{theorem}

\begin{theorem}[{\bf \cite[Theorem 2.4, Chapter 14]{GKZ}}]\label{thm: hyperdet}
The generating function of the degrees $N_d(k_1,\ldots, k_d)$ of the hyperdeterminants of format $\prod_{i=1}^d (k_i +1)$ is given by
\[
\sum_{k\in \mathbb N^d} N_d(k_1,\ldots, k_d) {\bf x}^{k} = \left[\sum_{i=0}^d (1-i) e_i({\bf x})\right]^{-2}\,,
\]
where $k=(k_1,\ldots,k_d)$ and $e_i({\bf x})$ is the $i$-th elementary function in the variables $x_1,\ldots,x_d$. 
\end{theorem}

\noindent Henceforth, we let $H({\bf x})\coloneqq\sum_{i=0}^d (1-i) e_i({\bf x})$ and $\mathcal V \coloneqq \left\lbrace H({\bf x}) = 0\right\rbrace$. 

\begin{lemma}\label{lem: partial H}
For all $i\in[d]$ we have $\partial_{ii}H({\bf x})=0$. If additionally we consider the point ${\bf c}=(\frac{1}{d-1}, \frac{1}{d-1}, \ldots, \frac{1}{d-1})$, then
\[
-\partial_{i_1\cdots i_k}H({\bf c}) = k\left(\frac{d}{d-1}\right)^{d-k-1}\quad \forall\ 1\le i_1<\cdots< i_k\le d\,.
\]
\begin{proof}
The first part follows immediately from the definition of $H$. For the second part, we have 
\[
-\partial_{i_1\cdots i_k}H({\bf c}) = \sum_{j=k}^d (j-1) e_{j-k}(\widehat{{\bf c}_{i_1\cdots i_k}}) = \sum_{i=0}^{d-k}(i+k-1)\binom{d-k}{i}\left(\frac{1}{d-1}\right)^i\,,
\]
where $\widehat{{\bf c}_{i_1\cdots i_k}}$ denotes all the variables $c_j$ except $c_{i_1},\ldots c_{i_k}$, and $e_{j-k}$ is the $(j-k)$-th elementary symmetric function in $d-k$ variables. Consider the $(d-1)\times (d-1)$ matrix $M=(M_{ij})_{i,j}$, where
\[
M_{ij} = \binom{d-k}{j-k+2}\left(\frac{1}{d-1}\right)^{j-k+2}\quad\forall\,0\le i,j\le d-2\,,
\]
where we use the convention $\binom{a}{b}=0$ if $b<0$. The matrix $M$ just defined has equal rows. We found this construction convenient for visualizing the sums. By a direct computation, one verifies that the sum of the entries strictly below the diagonal of $M$ is
\[
\sum_{i>j}M_{ij} = \sum_{i,j}M_{ij} - \sum_{i\le j}M_{ij} = (d-1)\sum_{j=0}^{d-2}M_{1j} - \sum_{i\le j}M_{ij} = (d-1)\left(\frac{d}{d-1}\right)^{d-k}+\partial_{i_1\cdots i_k}H({\bf c})\,.
\]

On the other hand, the sum of the entries strictly below the diagonal of $M$ is
\[
\begin{gathered}
\sum_{i>j}M_{ij} =
\sum_{j=0}^{d-3}(d-2-j)\binom{d-k}{j-k+2}\left(\frac{1}{d-1}\right)^{j-k+2}=
(d-k)\sum_{j=0}^{d-3}\binom{d-k-1}{j-k+2}\left(\frac{1}{d-1}\right)^{j-k+2}=\\
= (d-k)\sum_{j=k-2}^{d-3}\binom{d-k-1}{j-k+2}\left(\frac{1}{d-1}\right)^{j-k+2}=
(d-k)\sum_{s=0}^{d-k-1}\binom{d-k-1}{s}\left(\frac{1}{d-1}\right)^s=\\
=(d-k)\left(\frac{d}{d-1}\right)^{d-k-1}\,.
\end{gathered}
\]
Hence
\[
-\partial_{i_1\cdots i_k}H({\bf c}) = (d-1)\left(\frac{d}{d-1}\right)^{d-k} - (d-k)\left(\frac{d}{d-1}\right)^{d-k-1}=  k\left(\frac{d}{d-1}\right)^{d-k-1}\,.\qedhere
\]
\end{proof}
\end{lemma}

\begin{proposition}\label{prop: pointisgood}
The point ${\bf c} = (\frac{1}{d-1}, \frac{1}{d-1}, \ldots, \frac{1}{d-1})\in \mathcal V$ is smooth with $c_d\partial_d H({\bf c}) \neq 0$, strictly minimal, critical, isolated, and non-degenerate.
\begin{proof}
We first show that the point ${\bf c} = (\frac{1}{d-1}, \frac{1}{d-1}, \ldots, \frac{1}{d-1})$ sits in $\mathcal V$. This choice of ${\bf c}$ is the one made by Pantone in \cite{Pan}. We have
\[
\begin{gathered}
\sum_{i=0}^d(1-i)e_i({\bf c}) = \sum_{i=0}^d(1-i)\binom{d}{i}\left(\frac{1}{d-1}\right)^i=
\sum_{i=0}^d\binom{d}{i}\left(\frac{1}{d-1}\right)^i-\sum_{i=0}^di\binom{d}{i}\left(\frac{1}{d-1}\right)^i=\\
=\left(\frac{d}{d-1}\right)^d-\frac{d}{d-1}\sum_{i=1}^d\binom{d-1}{i-1}\left(\frac{1}{d-1}\right)^{i-1}=
\left(\frac{d}{d-1}\right)^d-\frac{d}{d-1}\sum_{j=0}^{d-1}\binom{d-1}{j}\left(\frac{1}{d-1}\right)^j=\\
=\left(\frac{d}{d-1}\right)^d-\left(\frac{d}{d-1}\right)^d=0\,,
\end{gathered}
\]
namely ${\bf c}$ belongs to $\mathcal V$. Moreover, by Lemma \ref{lem: partial H} we have that
\[
-\partial_d H({\bf c}) = \left(\frac{d}{d-1}\right)^{d-2}\neq 0\,,
\]
namely ${\bf c}$ is a smooth point of $\mathcal V$. The criticality of ${\bf c}$ follows from the symmetry of both ${\bf c}$ and $H$. Furthermore, the point ${\bf c}$ is strictly minimal and isolated by the proofs of \cite[Proposition 2.3, Proposition 2.6]{Pan}.

It remains to show that the point ${\bf c}$ is non-degenerate. By \cite[Definition 3.1]{RW}, the point ${\bf c}$ is non-degenerate whenever the quantity $\det \tilde{g}''(0)\neq 0$. By \cite[Proposition 4.2]{RW}, 
\[
\det\tilde{g}''(0) = d\,q^{d-1}\,,\quad q = 1 + c_1\frac{\partial_{dd} H({\bf c}) - \partial_{1 d} H({\bf c})}{\partial_d H({\bf c})}\,.
\]
By Lemma \ref{lem: partial H}, we have the identities
\[
\partial_{dd} H({\bf c})=0,\quad-\partial_{1d} H({\bf c}) = 2\left(\frac{d}{d-1}\right)^{d-3}\,.
\]
Substituting these values in the expressions for $q$ and $\det\tilde{g}''(0)$, we observe that ${\bf c}$ is non-degenerate:
\begin{align*}
& q = 1 - \left(\frac{1}{d-1}\right)\frac{2\left(\frac{d}{d-1}\right)^{d-3}}{\left( \frac{d}{d-1}\right)^{d-2}} = \frac{d-2}{d}\,,\\
& \det\tilde{g}''(0) = d\left(\frac{d-2}{d}\right)^{d-1} = \frac{(d-2)^{d-1}}{d^{d-2}}\neq 0\,.\qedhere
\end{align*} 
\end{proof}
\end{proposition}

\begin{proposition}\label{prop: L_0}
The following identity holds true:
\[
L_0(\tilde{u}_0, \tilde{g}) = \tilde{u}_0({\bf c}) = \frac{(d-1)^{2d-2}}{d^{2d-4}}\,.
\] 
\begin{proof}
The quantity $L_k(\tilde{u}_j,\tilde{g})$ is defined in the statement of \cite[Theorem 3.2]{RW}. For $k=0$, we have $L_0(\tilde{u}_j,\tilde{g}) = \tilde{u}_j({\bf c})$. 
By \cite[Proposition 4.3]{RW}, we have $\tilde{u}_0({\bf c}) = \frac{G({\bf c})}{(-c_d \partial_d H({\bf c}))^2}$, which gives the second equality in the statement.  
\end{proof}
\end{proposition}

\begin{theorem}\label{thm: hyperdetapprox}
Let $N(n{\bf 1}^d)$ be the degree of the hyperdeterminant of format $(n+1)^{\times d}$. Then asymptotically,  for $d\geq 3$ and $n\rightarrow +\infty$, 
\[
N(n{\bf 1}^d) \approx \frac{(d-1)^{2d-2}}{[2\pi(d-2)]^{\frac{d-1}{2}}\,d^{\frac{3d-6}{2}}}\cdot \frac{(d-1)^{dn}}{n^{\frac{d-3}{2}}}\,. 
\]
\begin{proof}
By Theorems \ref{thm: asympt} and \ref{thm: hyperdet}, for $N = 1$ and ${\bf c} = (\frac{1}{d-1},\ldots, \frac{1}{d-1})$, we have
\[
N(n{\bf 1}^d) \approx (d-1)^{dn} \left[\frac{1}{(2\pi n)^{\frac{d-1}{2}}\det\tilde{g}''(0)^{\frac{1}{2}}}\left(n\,L_0(\tilde{u}_0,\tilde{g}) + L_0(\tilde{u}_0, \tilde{g}) + L_0(\tilde{u}_1,\tilde{g})\right) + O\left(\frac{1}{n^{\frac{d-1}{2}}}\right)\right]\,,
\]
\noindent as $n\rightarrow \infty$. Define
\[
\eta_d\coloneqq\frac{1}{(2\pi)^{\frac{d-1}{2}}\det\tilde{g}''(0)^{\frac{1}{2}}}\,.
\]
Then
\[
N(n{\bf 1}^d) \approx (d-1)^{dn} \left[\frac{\eta_d\,L_0(\tilde{u}_0,\tilde{g})\,n}{n^{\frac{d-1}{2}}}+\frac{\eta_d(L_0(\tilde{u}_0,\tilde{g}) + L_0(\tilde{u}_1,\tilde{g}))}{n^{\frac{d-1}{2}}} + O\left(\frac{1}{n^{\frac{d-1}{2}}}\right)\right]\,.
\]
Since the second and third summand in the square brackets are both $O(n^{\frac{1-d}{2}})$, then
\begin{align*}
N(n{\bf 1}^d) &\approx (d-1)^{dn} \left[\frac{\eta_d\,L_0(\tilde{u}_0,\tilde{g})\,n}{n^{\frac{d-1}{2}}} + O\left(\frac{1}{n^{\frac{d-1}{2}}}\right)\right] \\
&= (d-1)^{dn} \left[\frac{\eta_d\,L_0(\tilde{u}_0,\tilde{g})\,n}{n^{\frac{d-1}{2}}} + \frac{n}{n}O\left(\frac{\eta_d\,L_0(\tilde{u}_0,\tilde{g})}{n^{\frac{d-1}{2}}}\right)\right] \\
&= (d-1)^{dn} \left[\frac{\eta_d\,L_0(\tilde{u}_0,\tilde{g})\,n}{n^{\frac{d-1}{2}}} + \frac{\eta_d\,L_0(\tilde{u}_0,\tilde{g})\,n}{n^{\frac{d-1}{2}}}O\left(\frac{1}{n}\right)\right]\\
&= \eta_d\,L_0(\tilde{u}_0,\tilde{g})\frac{(d-1)^{dn}}{n^{\frac{d-3}{2}}} \left[1+O\left(\frac{1}{n}\right)\right]\,.
\end{align*}
Conclusion follows by plugging in the identity for $L_0(\tilde{u}_0,\tilde{g})$ of Proposition \ref{prop: L_0}.
\end{proof}
\end{theorem}

\begin{remark}
For $d=3$ and $d=4$ we recover the approximations
\[
N(n{\bf 1}^3) \approx \frac{8^{n+1}}{3\sqrt{3}\pi}\left[1+O\left(\frac{1}{n}\right)\right],\quad
N(n{\bf 1}^4) \approx \frac{3^6}{2^9\pi\sqrt{\pi}}\frac{81^{n}}{\sqrt{n}}\left[1+O\left(\frac{1}{n}\right)\right]\,.
\]
The approximation formula for $d=3$ appears also in \cite[A176097]{Slo} and it is interesting to compare it with the asymptotic result proved  by Ekhad and Zeilberger for the EDdegree which is (see \cite{EZ,Pan})
\[
C(n{\bf 1}^3)\approx\frac{2}{\sqrt{3}\pi}\frac{8^{n}}{n}\left[1+O\left(\frac{1}{n}\right)\right]\,.
\]
\end{remark}

\begin{remark}[{\bf A Segre-Veronese hyperdeterminant}]\label{rmk: SegVer}
Theorem \ref{thm: hyperdetapprox} can be generalized further to a special class of Segre-Veronese varieties, with essentially the same calculations.
We preferred to keep the statement for hypercubical hyperdeterminants of Segre varieties for the ease of notation.

Given $\omega\in\Z_{\ge 0}$, the degree $\omega$ {\it Veronese embedding} of $\PP(W)$ is the image $v_{\omega}\PP(W)$ of the projective morphism
\begin{equation}\label{eq: Ver embedding}
v_{\omega}\colon \PP(W)\to\PP(S^{\omega} W)
\end{equation}
defined by $v_{\omega}([z])\coloneqq[z^{\omega}]$, where $S^{\omega}W$ is the degree $\omega$ symmetric power of $W$.

Segre-Veronese varieties are obtained combining the Veronese embedding $v_\omega$ already defined with the Segre embedding $\mathrm{Seg}$ introduced in Definition \ref{def: Segre embedding}. More precisely, let $(\omega_1,\ldots,\omega_d)\in\mathbb N^d$. The {\it degree $(\omega_1,\ldots,\omega_d)$ Segre-Veronese embedding} of $\PP(V_1)\times \cdots \times \PP(V_d)$ is the Segre embedding of the product $v_{\omega_1}\PP(V_1)\times\cdots\times v_{\omega_d}\PP(V_d)$, which we keep calling $v_{\omega_1}\PP(V_1)\times\cdots\times v_{\omega_d}\PP(V_d)$ for simplicity.

Now let $N(k_1,\ldots,k_d; \omega_1,\ldots,\omega_d)$ be the degree of the hyperdeterminant, namely the polynomial defining the dual hypersuperface, when defined,
of $v_{\omega_1}\PP(V_1)\times\cdots\times v_{\omega_d}\PP(V_d)$. Assume that $\omega_1=\cdots=\omega_d=\omega$ for some $\omega\in\mathbb{N}$. Applying \cite[Theorem 2.4]{GKZ}, one verifies that the generating function of the degrees $N_d(k_1,\ldots, k_d;\omega^d)$ is given by
\[
\sum_{k\in \mathbb N^d} N_d(k_1,\ldots,k_d;\omega^d) {\bf x}^{k} = \left[\sum_{i=0}^d (1-\omega i) e_i({\bf x})\right]^{-2}\,.
\]

In order to apply again Theorem \ref{thm: asympt}, a possible choice of a point is ${\bf c}=\left(\frac{1}{\omega d-1},\ldots,\frac{1}{\omega d-1}\right)$.
Then asymptotically, for $d\geq 3$ and $n\rightarrow +\infty$
\[
N(n{\bf 1}^d;\omega^d) \approx \frac{(\omega d-1)^{2d-2}}{[2\pi(\omega d-2)]^{\frac{d-1}{2}}\,\omega^{\frac{4d-5}{2}}\,d^{\frac{3d-6}{2}}}\,\frac{(\omega d-1)^{dn}}{n^{\frac{d-3}{2}}}\,. 
\]
\end{remark}

\begin{remark}[{\bf The case of the discriminant}]\label{rmk: discriminant}
Now we focus on a special case of Remark \ref{rmk: SegVer}, namely when $d=1$. Here, the number $N(n;\omega)$ is the degree of the discriminant polynomial $\Delta_{n,\omega}$, i.e., the degree of the dual variety of the Veronese variety $X=v_\omega\PP(V)$.
It is interesting to recall that when $\omega$ is even, denoting by $\|x\|^2$ the squared Euclidean norm of $x$, for any $f\in S^\omega V$ the polynomial
\begin{equation}\label{eq:charpoly}
\psi_f(\lambda)\coloneqq\Delta_{n,\omega}(f(x)-\lambda\|x\|^\omega)\,,
\end{equation}
where $\|x\|^\omega = \left(\|x\|^2\right)^{\omega/2}$, vanishes exactly at the eigenvalues of $f$ (see \cite[Theorem 2.23]{QZ}, \cite[Theorem 3.8]{Sod18}). This is a beautiful result first discovered by Liqun Qi, who called $\psi_f$ the {\it E-characteristic polynomial} of $f$. Its degree is equal to
\begin{equation}\label{eq: Frobenius ED degree Veronese}
\mathrm{EDdegree}_F(X)=
\begin{cases}
n+1 & \mbox{for $\omega=2$}\\
\frac{(\omega-1)^{n+1}-1}{\omega-2} & \mbox{for $\omega>2$}\,,
\end{cases}
\end{equation}
where the quantity $\mathrm{EDdegree}_F(X)$ is recalled in Definition \ref{def: EDdegree}.
The identity \eqref{eq: Frobenius ED degree Veronese} was proved in \cite[Theorem 5.5]{CS} and is a particular case of \cite[Theorem 12]{FO}. However, this result had already essentially been known in complex dynamics by the work of Forn{\ae}ss and Sibony \cite{FS}.

From equation \eqref{eq: Frobenius ED degree Veronese}, we see that $\mathrm{EDdegree}_F(X)$ is much less than the degree of $\Delta_{n,\omega}$, which is $N(n;\omega)=(n+1)(\omega-1)^n$. The reason of this huge degree drop in (\ref{eq:charpoly}) is that the polynomial $\|x\|^\omega$ (for $\omega>2$) defines a non-reduced hypersurface, hence highly singular, where $\Delta_{n,\omega}$ vanishes with high multiplicity: having a root of high multiplicity in (\ref{eq:charpoly}) for $\lambda=+\infty$ corresponds to a large degree drop.
Our asymptotic analysis for the hyperdeterminant originally arose from the desire to understand whether some analogous result could have been true in the non-symmetric setting. One has the ratio
\[
\frac{N(n;\omega)}{\mathrm{EDdegree}_F(X)}\approx\frac{\omega-2}{\omega-1}\,n\quad\mbox{as $n\to+\infty$}\,.
\]
One might also analyze the asymptotics of $N(n;\omega)$ and $\mathrm{EDdegree}_F(X)$ for $n$ fixed and $\omega\to+\infty$. In this case, we have the ratio
\[
\frac{N(n;\omega)}{\mathrm{EDdegree}_F(X)}\approx n+1\quad\mbox{as $\omega\to+\infty$}\,.
\]
Instead of $\mathrm{EDdegree}_F(X)$, we might consider the {\em generic ED degree} of $X$, recalled in Definition \ref{def: generic EDdegree}. It was shown in \cite[Proposition 7.10]{DHOST} that
\[
\mathrm{EDdegree}_{\mathrm{gen}}(X)=\frac{(2\omega-1)^{n+1}-(\omega-1)^{n+1}}{\omega}
\]
(see Remark \ref{rmk: gen EDdegree Segre} for a more general case).
From the last formula, we obtain the ratio
\[
\frac{N(n;\omega)}{\mathrm{EDdegree}_{\mathrm{gen}}(X)}\approx \frac{n+1}{2^{n+1}-1}\quad\mbox{as $\omega\to+\infty$}\,.
\]
\end{remark}

\subsection{Asymptotics for the hyperdeterminant of \texorpdfstring{$(\mathbb P^1)^{\times d}$}{P1 d}}\label{sec: anotherhyperdet}

The degree of the hyperdeterminant of $X=(\mathbb P^1)^{\times d}$ is denoted by $N({\bf 1}^d)$. Here we analyze the asymptotics with respect to $d$. By \cite[Corollary 2.10, Chapter 14]{GKZ}, the exponential generating function for $N({\bf 1}^d)$ is
\[
\sum_{d=0}^{+\infty}N({\bf 1}^d)\frac{x^d}{d!}=e^{-2x}(1-x)^{-2}\,. 
\]

\begin{proposition}\label{prop: binaryasympt}
	For $d\gg 0$, we have
	\[
	N({\bf 1}^d) \approx \sqrt{2\pi}\,\frac{d^{\frac{2d+1}{2}}(d+3)}{e^{d+2}}\,.
	\]
	\begin{proof}
		We have
		\[
		e^{-2x}(1-x)^{-2}=\sum_{i=0}^{+\infty}\frac{(-2)^i}{i!}x^i\sum_{j=0}^{+\infty}(j+1)x^j=\sum_{d=0}^{+\infty}\left[d!\sum_{i=0}^{d}\frac{(-2)^i}{i!}(d-i+1)\right]\frac{x^d}{d!}\,.
		\] 
		Therefore,
		\[
		\begin{gathered}
		N({\bf 1}^d) = d!\sum_{i=0}^{d}\frac{(-2)^i}{i!}(d-i+1) =\\
		= d!\left[(d+1)\sum_{i=0}^{d}\frac{(-2)^i}{i!}-\sum_{i=1}^{d}\frac{(-2)^i}{(i-1)!}\right] = d!\left[(d+1)\sum_{i=0}^{d}\frac{(-2)^i}{i!}+2\sum_{j=0}^{d-1}\frac{(-2)^j}{j!}\right]\,.
		\end{gathered}
		\]
		Both the inner sums converge to $e^{-2}$ as $d\to+\infty$. Using the Stirling approximation for the factorial
		\begin{equation}\label{eq: Stirling}
		d! \approx \sqrt{2\pi}\,\frac{d^{\frac{2d+1}{2}}}{e^d}\,, 
		\end{equation}
		we obtain the desired asymptotic formula. 
	\end{proof}
\end{proposition}

\begin{remark}
	Applying Theorem \ref{thm: FO formula}, one verifies that
	\[
	\mathrm{EDdegree}_F(X)=d!\,.
	\]
	Moreover, the generic ED degree of $X$ is (see Remark \ref{rmk: gen EDdegree Segre} for a more general formula)
	\begin{equation}\label{eq: gen EDdegree Segre P1}
	\mathrm{EDdegree}_{\mathrm{gen}}(X) = d!\sum_{i=0}^d\frac{(-2)^i}{i!}(2^{d+1-i}-1)
	= d!\left[2^{d+1}\sum_{i=0}^d\frac{(-1)^i}{i!}-\sum_{i=0}^d\frac{(-2)^i}{i!}\right]\,.
	\end{equation}
	The two inner sums converge to $e^{-1}$ and $e^{-2}$ as $d\to+\infty$, respectively. Using \eqref{eq: Stirling} we obtain the asymptotic formulas
	\[
	\mathrm{EDdegree}_F(X) \approx \sqrt{2\pi}\,\frac{d^{\frac{2d+1}{2}}}{e^d}\,,\quad\mathrm{EDdegree}_{\mathrm{gen}}(X) \approx \sqrt{2\pi}\,\frac{d^{\frac{2d+1}{2}}}{e^{d+2}}(2^{d+1}e-1)\,,
	\]
	so that
	\[
	\frac{N({\bf 1}^d)}{\mathrm{EDdegree}_F(X)}\approx \frac{d+3}{e^2}\,,\quad \frac{N({\bf 1}^d)}{\mathrm{EDdegree}_{\mathrm{gen}}(X)}\approx \frac{d+3}{e^2(2^{d+1}e-1)}\quad\mbox{as $d\to+\infty$}\,.
	\]
	In particular, the degree of the hyperdeterminant $N({\bf 1}^d)$ grows faster than $\mathrm{EDdegree}_F(X)$ and slower than $\mathrm{EDdegree}_{\mathrm{gen}}(X)$ as $d\to+\infty$.
\end{remark}

\section{Stabilization of the ED degree of some Segre varieties}\label{sec: stabilization}

Throughout the section, we let $V_1^\mR,\ldots,V_d^\mR$ be real vector spaces of dimensions $\dim(V_i^\mR)=n_i+1$, respectively. Recall from \S\ref{sec: prelim} that a real tensor of format $(n_1+1)\times \cdots \times (n_d+1)$ is  a multilinear map $t\colon (V_1^\mR)^*\times \cdots \times (V_d^\mR)^*\rightarrow \R$, i.e., an element of the tensor product (over $\mathbb R$) $V^\mR\coloneqq V_1^\mR\otimes\cdots\otimes V_d^\mR$. 

Suppose each $V_j^\mR$ comes equipped with an inner product (and so with a natural distance function). Their tensor product $V^\mR$ inherits a natural inner product defined as follows.

\begin{definition}\label{def: Frobenius}
The {\it Frobenius inner product} of two real decomposable tensors $t = x_1\otimes\cdots\otimes x_d$ and $t' = y_1\otimes\cdots\otimes y_d$ is
\begin{equation}\label{eq: Frobenius inner product for tensors}
q_F(t, t')\coloneqq q_1(x_1, y_1)\cdots q_d(x_d, y_d)\,,
\end{equation}
and it is naturally extended to every vector in $V^\mR$. When $V_j^\mR$ is equipped with the standard Euclidean inner product for all $j\in[d]$, one finds that
\begin{equation}\label{eq: Frobenius in coordinates for tensors}
q_F(t,t')=\sum_{i_1,\ldots,i_d}t_{i_1,\ldots,i_d}t'_{i_1,\ldots,i_d}
\end{equation}
for all $t,t' \in V^\mR$. The {\it (squared) distance function} is then $\delta_F(t,t')\coloneqq q_F(t-t', t-t')$ for all $t,t'\in V^\mR$.
\end{definition}

Analogously to what happens in a Euclidean space, it is natural to look at critical points of the distance function $\delta_F(t, \cdot)\colon X^\mR\to\R$ from a given tensor $t\in V^\mR$ with respect to some special sets $X^\mR\subset V^\mR$.
The most relevant for our purposes is the real affine cone over the Segre variety $X^\mR=\PP(V_1^\mR)\times\cdots\times\PP(V_d^\mR)$.
This leads us to the more general definition of ED degree of an affine variety to be discussed in a moment.

Let $V^\mR$ be a real vector space equipped with a distance function $\delta\colon V^\mR\times V^\mR\rightarrow \mathbb R$. Let $X^\mR\subset V^\mR$ be a real affine variety and let $u\in V$ be general. Consider the complex vector space $V\coloneqq V^\mR\otimes\C$ and the complex variety $X\coloneqq X^\mC$.
The distance function $\delta$ is extended to a complex-valued function $\delta\colon V\times V\to\C$ (which is not a Hermitian inner product).
The point is that even though the function $\delta$ is truly a distance function only over the reals, the complex critical points of $\delta$ on $X$ are important to draw all the metric information about the real affine cone $X^\mR$. 

\begin{definition}[ED degree \cite{DHOST}]\label{def: EDdegree}
The {\it Euclidean distance degree} ({\it ED degree}) of $X$ is the (finite) number of {\it complex} critical points of the function $\delta(u,\_)\colon X\setminus X_{\mathrm{sing}}\to \C$, where $X_{\mathrm{sing}}$ is the singular locus of $X$. We denote it by $\mathrm{EDdegree}_\delta(X)$ in order to stress its dependence on $\delta$.
\end{definition}

\begin{definition}[Generic ED degree]\label{def: generic EDdegree}
The {\em isotropic quadric} associated to $\delta$ is the quadric hypersurface $Q_\delta\coloneqq\{x\in V\mid \delta(v,v)=0\}$. When $X$ is transversal to $Q_\delta$ (this assumption holds for a general $Q_\delta$), the ED degree of $X$ with respect to $\delta$ is called {\em generic ED degree} of $X$ and it is denoted by $\mathrm{EDdegree}_\mathrm{gen}(X)$.
\end{definition}

In the following, we focus on the case when $V = V_1\otimes \cdots \otimes V_d$, $X=\PP(V_1)\times\cdots\times\PP(V_d)$ and $\delta = \delta_F$. We write $\mathrm{EDdegree}_F(X)$ to indicate $\mathrm{EDdegree}_{\delta_F}(X)$. The elements of the isotropic quadric $Q_{\delta_F}$ are called {\it isotropic tensors} of $V$.

Given a tensor $t\in V$, we refer to (complex) critical points of $\delta_{F}(t,\cdot)$ on $X$ simply as {\it critical points}. Lim \cite{L} and Qi \cite{Q} independently defined {\em singular vector tuples} of tensors and associated them to non-isotropic critical points of the distance function from the affine cone over the Segre variety $X^\mR$. The next result is a reformulation of \cite[Eq. (9)]{L} and of \cite[Lemma 19]{FO}.

\begin{theorem}\label{thm: singular vector tuple Lim Qi}
Given a real tensor $t\in V^\mR$, the non-isotropic decomposable critical points of $t$ correspond to tensors $v=\sigma\left(x^{(1)}\otimes\cdots\otimes x^{(d)}\right)\in V$ such that $q_j\left(x^{(j)}, x^{(j)}\right)=1$ for all $j\in[d]$ and 
\begin{equation}\label{eq: system singular vector tuple}
q_F\left(t, x^{(1)}\otimes\cdots\otimes  x^{(j-1)}\otimes\_\!\otimes x^{(j+1)}\otimes\cdots\otimes x^{(d)}\right)=\sigma\,q_j\left(x^{(j)},\_\right)\quad\forall\,j\in[d] 
\end{equation}
for some $\sigma\in\C$, called a {\it singular value} of $t$ corresponding to $v$. The corresponding $d$-ple $\left(x^{(1)},\ldots, x^{(d)}\right)$ is called {\it singular vector $d$-ple} of $t$. (In view of these results, we shall refer to singular vector tuples simply as {\it critical points}.)
\end{theorem}

\noindent For each $j\in[d]$, equation \eqref{eq: system singular vector tuple} may be written as
\begin{equation}\label{eq: system singular vector tuple, in coordinates}
t\left(x^{(1)}\otimes\cdots\otimes x^{(j-1)}\otimes x^{(j+1)}\otimes\cdots\otimes  x^{(d)}\right)=\sigma\,x_j\,, 
\end{equation}
where, on the left-hand side, we have a contraction of the tensor $t$ along the $j$-th direction. In analogy with matrices, the {\it best rank one approximation problem} for $t$ is solved as indicated in the following result by Lim \cite[Eq. (9)]{L} which we reformulate similarly to \cite[Theorem 20]{FO}.

\begin{theorem}\label{thm: Lim, variational problem}
Let $t\in V^\mR$ be a real tensor. Then $t$ admits real singular values and real critical points. Suppose $\widetilde{\sigma}$ is a real singular value of $t$ such that $\widetilde{\sigma}^2$ is maximum, and assume $\widetilde{v}=\widetilde{\sigma}\left(\widetilde{x}^{(1)}\otimes \cdots \otimes\widetilde{x}^{(d)}\right)$ is a critical point corresponding to $\widetilde{\sigma}$. Then $\widetilde{v}$ is a best rank one approximation of the tensor $t$. Moreover, a best rank one approximation of $t$ is unique if $t\in V$ is general.
\end{theorem}

The number of singular vector $d$-ples of a general tensor $t\in V$, i.e., the ED degree of the Segre variety $X=\PP(V_1)\times\cdots\times\PP(V_d)$ with respect to the distance function $\delta_F$ in $V^\mR$, is the content of the next result.

\begin{theorem}{\cite[Theorem 1]{FO}}\label{thm: FO formula}
	The ED degree of the Segre variety $X=\PP(V_1)\times\cdots\times\PP(V_d)\subset\PP(V)$ with respect to the Frobenius inner product $\delta_F$ in $V^\mR$ equals the coefficient of the monomial $h_1^{n_1}\cdots h_d^{n_d}$ in the polynomial
	\[
	\prod_{i=1}^d\frac{\widehat{h}_i^{n_i+1}-h_i^{n_i+1}}{\widehat{h}_i-h_i},\quad\widehat{h}_i\coloneqq\sum_{j\neq i}^dh_j\,.
	\]
\end{theorem}

Now assume that the tensor $t\in V$ is expressed in coordinates by the multidimensional array $(t_{i_1\cdots\,i_d})$, where $i_j\in[n_j+1]$ for all $j\in[d]$. Then the critical points of $t$ are of the form $\sigma (x^{(1)}\otimes \cdots \otimes x^{(d)})\in V$, with no zero component, and satisfy equations \eqref{eq: system singular vector tuple, in coordinates} which can be rewritten as 
\[
\sum_{i_\ell\in[n_l+1]}t_{i_1\cdots\,i_j\cdots\,i_d}\,x^{(1)}_{i_1}\cdots \widehat{x^{(j)}_{i_j}}\cdots x^{(d)}_{i_d}=\sigma\,x^{(j)}_{i_j}\quad \forall\,i_j\in[n_j+1]\,. 
\]
Eliminating the parameter $\sigma\in \C$, one derives the multilinear relations (for all $1\le k<s\le n_j+1$ and for all $j\in[d]$) that all critical points must satisfy: 

\begin{equation}\label{eq: equations singular space}
\sum_{i_\ell\in[n_l+1]}\left(t_{i_1\cdots\,k\,\cdots\,i_d}\,x^{(1)}_{i_1}\cdots x^{(j)}_s\cdots x^{(d)}_{i_d}-t_{i_1\cdots\,s\,\cdots\,i_d}\,x^{(1)}_{i_1}\cdots x^{(j)}_k\cdots x^{(d)}_{i_d}\right)=0\,.
\end{equation}

\begin{definition}[{\bf Critical space of a tensor}]\label{def: critical space}
The {\it critical space} (or {\it singular space}) $H_t$ of the tensor $t\in V$ is the linear projective space defined the equations (in the unknowns $z_{i_1\cdots i_d}$ that serve as linear functions on $V$)
\[
\sum_{i_\ell\in[n_l+1]}\left(t_{i_1\cdots\,k\,\cdots\,i_d}\,z_{i_1\cdots\,s\,\cdots\,i_d}-t_{i_1\cdots\,s\,\cdots\,i_d}\,z_{i_1\cdots\,k\,\cdots\,i_d}\right)=0\quad\forall\,1\le k<s\le n_j+1,\ \forall j\in[d]\,.
\]
\end{definition}

\begin{remark}
The tensor $t$ belongs to its critical space $H_t$ \cite[\S 5.2]{OP}.
For a general tensor $t$, let $Z_t$ denote the set of critical points and consider its projective span $\langle Z_t\rangle$.
Then $\langle Z_t\rangle \subset H_t$ and \cite[\S 3.5]{DOT} shows that they coincide for formats $(n_1+1)\times\cdots\times(n_d+1)$ satisfying a {\it triangle inequality} (i.e. the so-called {\it sub-boundary format}).
However, they do not concide in every format.
For the Segre variety of format $2\times 2\times 4$, $Z_t$ consists of $8$ critical points with $\langle Z_t\rangle \cong \PP^6$, but $H_t \cong \PP^7$. 
\end{remark}

\begin{lemma}\label{lem: singularvaluezero}
Let $X=\PP(V_1)\times\cdots\times\PP(V_d)\times \PP(W)$ with $\dim(W) = m+1$. Let $t\in V\otimes W$ be a {\it non-concise} tensor, i.e. there exists a proper subspace $L\subset W$ such that $t\in V\otimes L$. Then either a critical point is in $X\cap \PP(V \otimes L)$ or its singular value is zero.  
\begin{proof}
Fix bases for the vector spaces $V_i$ and a basis $\lbrace y_i\rbrace$ for $W$, and assume $\dim(L) = \ell+1\leq m+1$. By assumption, there exists a change of bases such that we may write the tensor $t\in V \otimes W$ as 
\[
t = \sum_{i=1}^{\ell} t_i\otimes y_i\,. 
\]
Now, let $v = x_1\otimes \cdots \otimes x_d\otimes z\in X$ be a critical point of $t$ with a non-zero singular value $\sigma\neq 0$. By their defining equations \eqref{eq: system singular vector tuple, in coordinates}, we have
\[
t(x_1\otimes \cdots \otimes x_d) = \sigma z\,.
\]
Denote by $\{y_j^*\}\subset W^*$ the dual basis of $\{y_j\}$. Since $t\in V \otimes L$, and since $z$ is the result of a contraction of $t$, this vector satisfies $y_j^{*}(z) = 0$ for all $\ell+1\leq j\leq m$. However, in the given basis, these are the defining equations of $\PP(L)\subset\PP(W)$ and so $v\in X\cap \PP(V \otimes L)$. 
\end{proof}
\end{lemma}

\begin{remark}
Let $d=3$ and $n_1=n_2=m=1$. A general tensor $t\in V\otimes W$ has $6$ distinct singular values. Assume that $t$ is non-concise, in particular $t\in V\otimes L$, where $L\subset W$ with $\dim(L)=1$. We consider the ED polynomial $\mathrm{EDpoly}_{X^\vee,t}(\varepsilon^2)$ of the dual variety of $X=\PP(V_1)\times\PP(V_2)\times\PP(W)$ at $t$ (see \S\ref{sec: edpolysegre} for the definition of ED polynomial). It turns out that the roots of $\mathrm{EDpoly}_{X^\vee,t}(\varepsilon^2)$ are the squared singular values of $t$ (see \cite[Proposition 5.1.4]{Sod20}). The second author computed symbolically in \cite[\S 5.4]{Sod20} the ED polynomial $\mathrm{EDpoly}_{X^\vee,u}(\varepsilon^2)$, for any $u\in V\otimes W$, as a univariate polynomial in $\varepsilon^2$ of degree $6$ whose coefficients are homogeneous polynomials in the entries $u_{ijk}$ of $u$. In particular, when $u=t$ and assuming that $V\otimes L$ has equations $t_{112}=t_{122}=t_{212}=t_{222}=0$, one verifies by direct computation that
\[
\mathrm{EDpoly}_{X^\vee,t}(\varepsilon^2) = c(t)\det(AA^T-\varepsilon^2I)\,\varepsilon^8
\]
where $c(t)$ is a homogeneous polynomial in $t_{ijk}$ and $A\in V$ is the $2\times 2$ slice $(t_{ij1})$ of $t$. On one hand, $\det(AA^T-\varepsilon^2I)=\mathrm{EDpoly}_{Y^\vee,A}(\varepsilon^2)$, where
\[
Y=X\cap\PP(V\otimes L)=\PP(V_1)\times\PP(V_2)\times\PP(L)\cong\PP^1\times\PP^1\subset\PP^3\,,
\]
i.e., the variety of $2\times 2$ matrices of rank one. In particular, two of the singular values of $t$ are non-zero and correspond to critical points in $Y$. On the other hand, the factor $\varepsilon^8=(\varepsilon^2)^4$ tells us that the remaining $4$ singular values of $t$ are zero, thus confirming Lemma \ref{lem: singularvaluezero}.
\end{remark}

\begin{example}
Keep the notation from Lemma \ref{lem: singularvaluezero}.
It is possible that a given non-concise tensor $t$ possesses critical points {\it outside} $X\cap \PP(V\otimes L)$ as shown by the next example.
Let $d=3$, $n_1=n_2=1$ and $m=2$. Consider the following non-concise tensor $t\in V\otimes L\subset V\otimes W$, where $L$ is the hyperplane $L =\lbrace z_2=0\rbrace$:
\[
t = 4\,x_0y_0z_0+0.1\,x_1y_0z_0+0.556\,x_0y_1z_0+2.5\,x_1y_1z_0+2\,x_0y_0z_1+2.667\,x_1y_0z_1+x_0y_1z_1+x_1y_1z_1\,.
\]

We list in Table \ref{tab: triples 2x2x3 tensor} the $8$ singular vector triples $[(a_0,a_1),(b_0,b_1),(c_0,c_1,c_2)]$ of $t$, together with their corresponding singular values $\sigma$. Observe that $t$ possesses two critical points with singular value zero in the orthogonal $L^{\perp} = \langle z_2\rangle$.

\begingroup
\setlength{\tabcolsep}{5pt}
\renewcommand{\arraystretch}{1.15}
\begin{table}[ht]
	\centering
	\begin{tabular}{|c||c|c|c|c|c|c|c|c|}
	\hline
	$a_0$ &0.844&0.222&0.980&0.194&0.785+0.443\,$\sqrt{-1}$&0.785-0.443\,$\sqrt{-1}$&
	0.755&0.892\\
	$a_1$ &0.536&-0.975&-0.200&-0.981&-0.862+0.404\,$\sqrt{-1}$&-0.862-0.404\,$\sqrt{-1}$&-0.656&0.452\\
	\hline
	$b_0$ &0.898&0.275&0.077&0.974&0.846+0.265\,$\sqrt{-1}$&0.846-0.265\,$\sqrt{-1}$&0.382&0.410\\
	$b_1$ &0.440&-0.961&-0.997&-0.225&-0.681+0.329\,$\sqrt{-1}$&-0.681-0.329\,$\sqrt{-1}$&0.924&-0.912\\\hline
	$c_0$ &0.751&0.999&0.356&0.512&0.995+0.146\,$\sqrt{-1}$&0.995-0.146\,$\sqrt{-1}$&0&0\\
	$c_1$ &0.660&0.054&-0.935&-0.859&-0.401+0.361\,$\sqrt{-1}$&-0.401-0.361\,$\sqrt{-1}$&0&0\\
	$c_2$ &0&0&0&0&0&0&1&1\\\hline
	$\sigma$ &5.161&2.446&-0.715&-2.321&-2.946-0.495\,$\sqrt{-1}$&2.946-0.495\,$\sqrt{-1}$&0&0\\
	\hline
	\end{tabular}
	\vspace*{1mm}
	\caption{Singular vector triples and singular vectors of a non-concise $2\times2\times3$ tensor.}\label{tab: triples 2x2x3 tensor}
\end{table}
\endgroup

For the convenience of the reader, here we collect a piece of code we used to compute numerically the singular vector triples and the singular values of a format $2\times2\times3$ tensor shown in Table \ref{tab: triples 2x2x3 tensor}, which inspired our results of \S\ref{sec: stabilization}. The following code is written in the software \texttt{Macaulay2}. 
\begin{small}
\begin{Verbatim}[commandchars=\\\{\}]
R = CC[a0,a1,b0,b1,c0,c1,c2,sigma];
aa = matrix\{\{a0,a1\}\}; bb = matrix\{\{b0,b1\}\}; cc = matrix\{\{c0,c1,c2\}\};
for i to 1 do for j to 1 do for k to 2 do t_(i,j,k) = 100*random(CC)
for i to 1 do for j to 1 do t_(i,j,3) = 0
T = sum(2, i -> sum(2, j -> sum(3, k -> t_(i,j,k)*aa_(0,i)*bb_(0,j)*cc_(0,k))));
\textcolor{blue}{-- I is the ideal of the singular vector triples and the singular values of T}
I = ideal(apply(2, i -> sub(contract(aa_(0,i),T),aa_(0,i)=>0)-sigma*aa_(0,i))|
	apply(2, i -> sub(contract(bb_(0,i),T),bb_(0,i)=>0)-sigma*bb_(0,i))|
	apply(3, i -> sub(contract(cc_(0,i),T),cc_(0,i)=>0)-sigma*cc_(0,i))|
	{sum(2, i-> aa_(0,i)^2)-1, sum(2, i -> bb_(0,i)^2)-1, sum(3, i -> cc_(0,i)^2)-1});
H = first entries gens I;
\textcolor{blue}{-- Now we compute numerically the zeros of I using PHCpack}
needsPackage "PHCpack";
elapsedTime solutions = solveSystem H;
triples = apply(#solutions, j -> (solutions#j).Coordinates);
\end{Verbatim}
\end{small}
The list \verb+triples+ stores the eight singular vector triples of \verb+T+ (up to sign) and their singular values.
\end{example}

Recall that the Segre variety $\PP^{n_1}\times \cdots \times \PP^{n_d}\times \PP^N$, where  $N = \sum_{i=1}^{d} n_i$, and the corresponding hyperdeterminant are said to be of {\it boundary format}. This format turns out to be important for our purposes. 

Keeping the notation from above, let $\dim(V_i) = n_i + 1$ and $\dim(W) = m+1$. Let $t\in V\otimes L\subset V\otimes W$, where $L\subset W$ is a hyperplane, namely $t$ is non-concise and has the {\it last slice} zero. Then the hyperdeterminant vanishes on $t$, i.e. $\mathrm{Det}(t)=0$. 

By definition, this means that the tensor $t$ is {\it degenerate}: there exists a non-zero decomposable tensor $v_1\otimes \cdots \otimes v_d \otimes z\in V\otimes W$ such that 
\begin{equation}\label{eq: kerneloft}
t(v_1,\ldots, V_i, \ldots, v_{d}, z) = 0\quad\forall\,i\in[d]\quad\mbox{and}\quad t(v_1,\ldots, v_{d}, W) = 0\,. 
\end{equation}

The {\it kernel} $K(t)$ of a tensor $t$ is the variety of all non-zero $v_1\otimes \cdots \otimes v_{d} \otimes z\in V\otimes W$
such that \eqref{eq: kerneloft} is satisfied. The description of the critical points of $t$ outside the hyperplane $L$ can be given in terms of $K(t)$, this is the content
of our next result.

\begin{theorem}\label{thm: descriptionofcriticalpoints}
 Let $\dim(V_i) = n_i + 1$ and $\dim(W) = m+1$, $N = \sum_{i=1}^{d}n_i$ and $m\ge N$.
Let $t\in V\otimes W$ be a tensor such that the flattening map
\begin{equation}\label{eq:flattening}
\pi_{W}\colon V\to W^{*}
\end{equation}
has rank $N$, i.e., $t\in V\otimes L$ where $L\subset V$ is a subspace of dimension $N$.
Assume that $t$ is general with this property. Then
\begin{enumerate}
\item[$(i)$] the kernel $K(t)$ of $t$ consists of $\frac{N!}{\prod_i n_i!}$ linear spaces of projective dimension
$m-N$ corresponding to the intersection of the kernel of the flattening map (\ref{eq:flattening})
with the  Segre variety of rank one matrices $X=\PP(V_1)\times\cdots\times\PP(V_{d})$, which has degree
$\frac{N!}{\prod_i n_i!}$,

\item[$(ii)$] the points of $K(t)$ are exactly the critical points of $t$ with zero singular value. Moreover, the latter critical points of $t$ are the only ones not lying on $V\otimes L$. In fact, they lie on its orthogonal complement $V\otimes L^\perp$. 
\end{enumerate}
\end{theorem}

\begin{proof}
\begin{enumerate}
	\item[$(i)$] The projectivization $\PP(\mathrm{Ker}(\pi_{W}))$ of the kernel of the flattening map \eqref{eq:flattening} has codimension $N$ in $\PP(V)$. By the genericity assumption on $t$, the intersection $\PP(\mathrm{Ker}(\pi))\cap X$ is given by $\deg(X) = \frac{N!}{\prod_i n_i!}$ points. 
	
	For each one of these points $v_1\otimes\ldots\otimes v_{d}$, consider the linear equations in the unknown $z\in W$
	\begin{equation}\label{linearsystker}
	t(v_1,\ldots, V_i, \ldots, v_{d}, z) = 0\quad\forall\,i\in[d]\,.
	\end{equation}
	These are $N$ linear homogeneous equations, which define a linear subspace of $\PP(W)$ of projective dimension
	$m-N$ \cite[Theorem 3.3(i)]{Ott}. Note that the decomposable tensors satisfying the linear system \eqref{linearsystker} are critical points. On the other hand, they are points of $K(t)$ by definition of the kernel of a tensor.
	
	\item[$(ii)$] By Lemma \ref{lem: singularvaluezero}, each critical point of $t$ is either in $V\otimes L$ or it has zero singular value. On the other hand, 
	since the tensor $t$ is general in $V\otimes L$, we may assume $\mathrm{Det}(t)\neq 0$. This implies that every critical point of $t$
	in  $V\otimes L$ has a non-zero singular value. 
	As a consequence, the only critical points outside $V\otimes L$ are the ones with singular value zero. 
	
	To see where they are located and thus establishing the last sentence, we proceed as follows. Since our vector spaces $V$ and $W$ are equipped with an inner product, they come with an indentification with their duals:  $V\cong V^{*}$ and $W\cong W^{*}$. Therefore the flattening map $\pi_{W}$ above may be regarded as a linear map $\pi_{W}\colon V\rightarrow W$. Note that the other flattening map $\pi_{V}\colon W\rightarrow V$ induced by $t$ is dual to $\pi_{W}$. In bases, this amounts to say that $\pi_{V} = \pi_{W}^T$. 
	
	Now, suppose that $t$ has a critical point $v\otimes z\in V\otimes W$ with singular value $\sigma$, where $v = v_1\otimes \cdots \otimes v_{d}$ is a decomposable tensor and $z\in W$. By definition, this means that $v$ and $z$ are non-zero with $\pi_{W}(v) = \sigma z$ and $\pi_{W}^T(z) = \sigma v$. Assume $z\notin L$. 
	As noticed above, this critical point has singular value $\sigma = 0$. 
	Since $t\in V\otimes L$, we have $\mathrm{Im}(\pi_{W}) = L$. Note that $\mathrm{Ker}(\pi_{W}^T) = \mathrm{Im}(\pi_{W})^{\perp} = L^{\perp}$. Since $\pi_{W}^T(z) = \sigma v = 0$, one finds $z\in\mathrm{Ker}(\pi_{W}^T) = L^{\perp}$. 
	In conclusion, the critical points of a general tensor $t$, that are outside $V\otimes L$, lie on its orthogonal complement.\qedhere
\end{enumerate} 
\end{proof}

\noindent Using directly equations \eqref{eq: system singular vector tuple, in coordinates} satisfied by the critical points, we also show the following result.

\begin{theorem}\label{thm: specialization}
Keep the notation from Theorem \ref{thm: descriptionofcriticalpoints}. Let $\mathrm{Det}$ be the hyperdeterminant in the boundary format $(n_1+1)\times \cdots \times(n_d+1) \times (N+1)$. Consider a tensor $t\in V\otimes\C^{N+1}\subset V\otimes W$ with $\mathrm{Det}(t)\neq 0$. Then the critical points of $t$ on the Segre product $\PP(V_1)\times\cdots\times\PP(V_{d})\times\PP(W)$ lie in the subvariety $\PP(V_1)\times\cdots\times\PP(V_{d})\times\PP(\C^{N+1})$.
\end{theorem}
\begin{proof}
We show that the critical space of $t$ in $\PP(V\otimes W)$ lies inside $\PP(V\otimes\C^{N+1})$. We may assume $\dim(W) = N+2$. 

Consider the equations \eqref{eq: system singular vector tuple, in coordinates} for $j=d$, $s= N+2$. Since by assumption $t_{{i_1}\cdots{i_{d}},N+2}=0$, these relations simplify and we obtain: 
\[
\left[\sum_{i_\ell\in[n_l+1]} t_{i_1\cdots i_{d},k}\,x^{(1)}_{i_1}\cdots x^{(d)}_{i_{d}} \right]x^{(d+1)}_{N+2}=0\quad\forall k\in[N+2]\,.
\]
The equations inside the brackets have no non-zero solutions by the assumption $\mathrm{Det}(t)\neq 0$ and by the description in \cite[Chapter 14, Theorem 3.1]{GKZ}. Hence $x^{(d+1)}_{N+2}=0$, which proves the statement.
\end{proof}

\begin{corollary}\label{cor: stabED}
Let $N = \sum_{i=1}^{d}n_i$. For all $m\geq N$, we have
\begin{equation}\label{eq: stabED}
\mathrm{EDdegree}_F\left(\PP^{n_1}\times\cdots\times\PP^{n_{d}}\times\PP^{m}\right) = \mathrm{EDdegree}_F\left(\PP^{n_1}\times\cdots\times\PP^{n_{d}}\times\PP^{N}\right)\,.
\end{equation}
\begin{proof}
Note that, given a smooth projective variety $X\subset \PP(V)$, when the number of critical points (of the distance function with respect to $X$) of a given point $t\in V$ is finite, it coincides with $\mathrm{EDdegree}(X)$. 
In other words, we may consider convenient specializations in order to compute the ED degree. To evaluate the ED degree on the left-hand side of \eqref{eq: stabED}, we specialize $t$ as in Theorem \ref{thm: specialization}. Then the result of Theorem \ref{thm: specialization} shows that the critical points of $t$ are the same as the ones needed to compute the right-hand side of \eqref{eq: stabED}. 
\end{proof}
\end{corollary} 

Theorem \ref{thm: specialization} generalizes to the partially symmetric case. The proof is analogous, following the critical space in the partially symmetric case, as defined in \cite{DOT}.

\begin{theorem}\label{thm: specialization-partial}
Keep the notation from Theorem \ref{thm: descriptionofcriticalpoints}. Let $\omega_i\in\Z_{\ge 0}$ for all $i\in[d]$. Let $\mathrm{Det}$ be the hyperdeterminant in the boundary format space
  $S^{\omega_1}V_1\otimes \cdots \otimes S^{\omega_{d}}V_{d}\otimes\C^{N+1}$.
  Consider a tensor $t\in S^{\omega_1}V_1\otimes \cdots \otimes S^{\omega_{d}}V_{d}\otimes\C^{N+1}\subset
  S^{\omega_1}V_1\otimes \cdots \otimes S^{\omega_d}V_{d}\otimes\C^{m+1}$ with $\mathrm{Det}(t)\neq 0$. Then the critical points of $t$ in $v_{\omega_1}\PP(V_1)\times\cdots\times v_{\omega_{d}}\PP(V_d)\times\PP(\C^{m+1})$ lie in the subvariety $v_{\omega_1}\PP(V_1)\times\cdots\times
  v_{\omega_{d}}\PP(V_d)\times\PP(\C^{N+1})$.
\end{theorem}

\begin{corollary}\label{cor: stabED-partial}
Let $N = \sum_{i=1}^{d}n_i$ and $\omega_i\in\Z_{\ge 0}$ for all $i\in[d]$. For all $m\geq N$, we have
\[
\mathrm{EDdegree}_F\left(v_{\omega_1}\PP^{n_1}\times\cdots\times v_{\omega_{d}}\PP^{n_{d}}\times\PP^{m}\right) = \mathrm{EDdegree}_F\left(v_{\omega_1}\PP^{n_1}\times\cdots\times v_{\omega_{d}}\PP^{n_{d}}\times\PP^{N}\right)\,.
\]
\end{corollary}

\begin{remark}
Corollary \ref{cor: stabED-partial} does not hold if the last factor $\PP^m$ is replaced by $v_s\PP^m$ for some integer $s\ge 2$.
This can be checked also from the formula in \cite[Theorem 12]{FO} (generalizing Theorem \ref{thm: FO formula} to the case of partially symmetric tensors) which does not stabilize anymore for $m\to +\infty$.
\end{remark}

\begin{conjecture}\label{conj: stabthm}
Let $X=\PP(V_1)\times\cdots\times\PP(V_{d})\subset\PP(V)$ and let $W$ be an $(m+1)$-dimensional complex vector space $W$, where $m\ge 0$. Then
\begin{equation}\label{eq: formula EDdegree sum}
\mathrm{EDdegree}_F(X\times\PP(W))=\sum_{j=0}^{m}\mathrm{EDdegree}_F(X\cap H_j)\,, 
\end{equation}
where $H_j\subset\PP(V)$ is a general subspace of codimension $j$.
\end{conjecture}

\begin{example}
As an illustration of Conjecture \ref{conj: stabthm}, we consider the Segre products $X\times\PP^m$ for $m\ge 0$ and for some choices of $X$. The values of $\mathrm{EDdegree}_F(X\times\PP^m)$ are listed in Table \ref{table: Frob ED degrees X times PP^m} for the different varieties $X$. The entries of the $i$-th column in Table \ref{table: Frob ED degrees X cap H_j} correspond to the values of $\mathrm{EDdegree}_F(X\cap H_{i-1})$. The numbers in Table \ref{table: Frob ED degrees X times PP^m} as well as in the first column of Table \ref{table: Frob ED degrees X cap H_j} are computed according to Theorem \ref{thm: FO formula}. The boxed ED degrees in Table \ref{table: Frob ED degrees X cap H_j} have been checked numerically with the software \texttt{Julia} \cite{BT}, and the remaining ones with the software \texttt{Macaulay2} \cite{GS}.

Observe that each number in the $i$-th column of Table \ref{table: Frob ED degrees X times PP^m} is the sum of the first $i$ entries in the corresponding row of Table \ref{table: Frob ED degrees X cap H_j}, thus confirming Conjecture \ref{conj: stabthm}.
\begingroup
\setlength{\tabcolsep}{3pt}
\renewcommand{\arraystretch}{1.15}
\begin{table}[htbp]
\centering
\begin{tabular}{|c||c|c|c|c|c|c|}
\hline
$X$ & $X\times\PP^0$ & $X\times\PP^1$ & $X\times\PP^2$ & $X\times\PP^3$ & $X\times\PP^4$ & $X\times\PP^5$\\
\hhline{|=||======|}
$\PP^1\times\PP^1$&$2$&$6$&$8$&$8$&$8$&$8$\\
$\PP^1\times\PP^2$&$2$&$8$&$15$&$18$&$18$&$18$\\
$\PP^2\times\PP^2$&$3$&$15$&$37$&$55$&$61$&$61$\\
$\PP^2\times\PP^3$&$3$&$18$&$55$&$104$&$138$&$148$\\
\hline
\end{tabular}
\vspace*{1mm}
\caption{Values of $\mathrm{EDdegree}_F(X\times\PP^m)$ for different choices of $m$ and $X$.}\label{table: Frob ED degrees X times PP^m}
\end{table}

\begin{table}[htbp]
	\centering
	\begin{tabular}{|c||c|c|c|c|c|c|}
		\hline
		$X$ & $X\cap H_0$ & $X\cap H_1$ & $X\cap H_2$ & $X\cap H_3$ & $X\cap H_4$ & $X\cap H_5$\\
		\hhline{|=||======|}
		$\PP^1\times\PP^1$&$2$&$4$&$2$&$0$&$0$&$0$\\
		$\PP^1\times\PP^2$&$2$&$6$&$7$&$3$&$0$&$0$\\
		$\PP^2\times\PP^2$&$3$&$\boxed{12}$&$\boxed{22}$&$\boxed{18}$&$\boxed{6}$&$0$\\
		$\PP^2\times\PP^3$&$3$&$\boxed{15}$&$\boxed{37}$&$\boxed{49}$&$\boxed{34}$&$\boxed{10}$\\
		\hline
	\end{tabular}
	\vspace*{1mm}
	\caption{Values of $\mathrm{EDdegree}_F(X\cap H_j)$ for different choices of $X$ and $j$.}\label{table: Frob ED degrees X cap H_j}
\end{table}
\endgroup
\end{example}

\begin{remark}
Using the notations of Theorem \ref{thm: descriptionofcriticalpoints}, we assume $d=2$. Let $t\in V\otimes L$, where $L\subset W$ is a linear subspace. The critical points of $t$ generally fill up several components of different dimensions, forming the {\it critical locus}. These components are either in $V\otimes L$ or in $V\otimes L^\perp$. By the description of the critical points in Theorem \ref{thm: descriptionofcriticalpoints}, the critical locus of a general $t$ sitting inside $V\otimes L^\perp$ coincides with the {\it contact locus} of $t$, see \cite[\S 3]{Ott} (these last two observations apply to all formats with any number of factors). For a general tensor $t\in V\otimes L$, define $\mathcal{C}_L$ and $\mathcal{C}_{L^\perp}$ to be the critical loci inside $V\otimes L$ and $V \otimes L^\perp$, respectively. In Table \ref{tab: dim deg critical loci} we collect the dimensions and the degrees of these loci for the first few cases of boundary formats and where the linear subspace $L$ is varying.

\begingroup
\setlength{\tabcolsep}{3pt}
\renewcommand{\arraystretch}{1.15}
\begin{table}
	\centering
	\begin{tabular}{|c|c|c|c|}
		\hline
		$(n_1,n_2,m)$ & $\dim(L)$ & $(\dim(\mathcal{C}_L),\deg(\mathcal{C}_L))$ & $(\dim(\mathcal{C}_{L^\perp}), \deg(\mathcal{C}_{L^\perp}))$\\
		\hhline{|====|}
		(1,1,2) & 3 & (0,8) & $\mathcal{C}_{L^\perp}=\emptyset$ \\
		\blue{(1,1,2)} & \blue{2} & \blue{(0,6)} & \blue{(0,2)} \\
		(1,1,2) & 1 & (0,2) & (2,2) \\
		\hline
		(1,2,3) & 4 & (0,18) & $\mathcal{C}_{L^\perp}=\emptyset$ \\
		\blue{(2,3,4)} & \blue{3} & \blue{(0,15)} & \blue{(0,3)} \\
		(1,2,3) & 2 & (0,8) & (2,3) \\
		(1,2,3) & 1 & (0,2) & (4,2) \\
		\hline
		(2,2,4) & 5 & (0,61) & $\mathcal{C}_{L^\perp}=\emptyset$ \\
		\blue{(3,3,5)} & \blue{4} & \blue{(0,55)} & \blue{(0,6)} \\
		(2,2,4) & 3 & (0,37) & (2,6) \\
		(2,2,4) & 2 & (0,15) & (4,4) \\
		(2,2,4) & 1 & (0,3) & (6,2) \\
		\hline
		(2,3,5) & 6 & (0,148) & $\mathcal{C}_{L^\perp}=\emptyset$ \\
		\blue{(3,4,6)} & \blue{5} & \blue{(0,138)} & \blue{(0,10)} \\
		(2,3,5) & 4 & (0,104) & (2,10) \\
		(2,3,5) & 3 & (0,55) & (4,7) \\
		(2,3,5) & 2 & (0,18) & (6,4) \\
		(2,3,5) & 1 & (0,3) & (8,2) \\
		\hline
	\end{tabular}
	\vspace*{1mm}
	\caption{Dimensions and degrees of the critical loci $\mathcal{C}_L$ and $\mathcal{C}_{L^\perp}$ for boundary formats. The blue dimension is when $L$ is a hyperplane.}\label{tab: dim deg critical loci}
\end{table}
\endgroup

\end{remark}

\begin{remark}\label{rmk: gen EDdegree Segre}
One might also compute the ED degree of a Segre-Veronese product of projective spaces $X=\omega_1\PP^{n_1}\times\cdots\times\omega_d\PP^{n_d}$, with respect to a metric that makes $X$ transversal to the isotropic quadric (i.e., the {\it generic ED degree} of $X$). The next formula is obtained by a Chern class computation made by the second author \cite{Sod20} and applying \cite[Theorem 5.8]{DHOST} (here $N=\dim(X)=n_1+\cdots+n_d$):
\begin{equation}\label{eq: gen EDdegree Segre}
\mathrm{EDdegree}_{\mathrm{gen}}(X)=\sum_{j=0}^N(-1)^j(2^{N+1-j}-1)(N-j)!\left[\sum_{i_1+\cdots+i_d=j}\prod_{l=1}^d\frac{\binom{n_l+1}{i_l}\omega_l^{n_l-i_l}}{(n_l-i_l)!}\right]\,.
\end{equation}
For example, if $d=2$, $n_1=\omega_1=\omega_2=1$ and $n_2=n$, then the identity \eqref{eq: gen EDdegree Segre} simplifies to
\begin{align*}
\mathrm{EDdegree}_{\mathrm{gen}}(\PP^1\times\PP^n) &=\sum_{i=0}^{n+1}(-1)^i(2^{n+2-i}-1)(n+1-i)!\left[\frac{\binom{n+1}{i}}{(n-i)!}+\frac{2\binom{n+1}{i-1}}{(n+1-i)!}\right]\\
&=\overbrace{\sum_{i=0}^{n+1}(-1)^i\,2^{n+2-i}(n+1-i)\binom{n+1}{i}}^A-\overbrace{\sum_{i=0}^{n+1}(-1)^i(n+1-i)\binom{n+1}{i}}^B\\
&\quad+2\overbrace{\sum_{i=0}^{n+1}(-1)^i\,2^{n+2-i}\binom{n+1}{i-1}}^C-2\overbrace{\sum_{i=0}^{n+1}(-1)^i\binom{n+1}{i-1}}^D\,,
\end{align*}
where one might check easily that
\[
A = 4(n+1)\,,\quad B = 0\,,\quad C = (-1)^{n+1}-1\,,\quad D = (-1)^{n+1}\,.
\]
Therefore
\[
\mathrm{EDdegree}_{\mathrm{gen}}(\PP^1\times\PP^n) = A-B+2C-2D = 4(n+1)-2 = 4n+2\,.
\]
In particular $\mathrm{EDdegree}_{\mathrm{gen}}(\PP^1\times\PP^n)$ diverges when $n\to+\infty$, in contrast with $\mathrm{EDdegree}_F(\PP^1\times\PP^n)=2$ for all $n\ge 1$.
\end{remark}

\section{Stabilization of the degree of the dual of a special Segre product}\label{sec: stabilization degree dual of XxQ}

In this section, we start introducing classical material on dual varieties. We refer to \cite{GKZ} for details on the rich theory of projective duality. 
We shall demonstrate Theorem \ref{thm: stabilization XxQ}, showing a stabilization property of dual varieties to some Segre products. 

\begin{definition}[{\bf Segre products}]\label{def: segprod}
Let $X_1\subset \PP(V_1)$ and $X_2\subset \PP(V_2)$ be two projective varieties. Their direct product $X_1\times X_2$ may be embedded in $\PP(V_1\otimes V_2)$ via the Segre embedding introduced in Definition \ref{def: Segre embedding}. The image of this embedding is called the {\it Segre product} of $X_1$ and $X_2$ and it is again denoted by $X_1\times X_2$. 
\end{definition}

\begin{definition}\label{def: conormal variety}
Let $Y\subset\PP(V)\cong\PP^n$ be an irreducible projective variety of dimension $m$, where $V$ is a Euclidean space. The Euclidean structure of $V$ allows us to naturally identify $V$ with its dual $V^*$. The {\it conormal variety} of $Y$ is the incidence correspondence
\[
\mathcal{N}_Y\coloneqq\overline{\{(z_1,z_2)\in V\times V\mid z_1\in Y_\mathrm{sm}\ \mbox{and}\ z_2\in N_{z_1}Y\}}\,,
\]
where $N_{z_1}Y$ denotes the {\it normal space} of $Y$ at the smooth point $z_1$.
\end{definition}

A fundamental feature of the conormal variety is the content of the {\it biduality theorem} \cite[Chapter 1]{GKZ}: one has $\mathcal{N}_Y=\mathcal{N}_{Y^\vee}$. The latter implies $(Y^\vee)^\vee=Y$, the so-called {\it biduality}.  

The {\it polar classes} of $Y$ are defined to be the coefficients $\delta_i(Y)$ of the class in cohomology
\[
[\mathcal{N}_{Y}]=\delta_0(Y)s^{n}t+\delta_1(Y)s^{n-1}t^2+\cdots+\delta_{n-1}(Y)st^{n}\in A^*(\PP(V)\times\PP(V))\cong\frac{\Z[s,t]}{(s^{n+1},t^{n+1})}\,,
\]
where $s=\pi_1^*([H])$, $t=\pi_2^*([H'])$, the maps $\pi_1,\pi_2$ are the projections onto the factors of $\PP(V)\times\PP(V)$ and $H, H'$ are hyperplanes in $\mathbb{P}(V)$.
If we assume $Y$ smooth, $\delta_i(Y)$ may be computed utilizing the Chern classes of $Y$.
These are the Chern classes of the tangent bundle $\mathcal{T}_Y$ of $Y$. One computes \cite[\S 3]{Hol}:

\begin{equation}\label{eq: Holme}
\delta_i(Y)=\sum_{j=0}^{m-i}(-1)^{j}\binom{m+1-j}{i+1}c_{j}(Y)\cdot h^{m-j} = \sum_{j=0}^{m-i}(-1)^{j}\binom{m+1-j}{i+1}\deg(c_{j}(Y))\,. 
\end{equation}
The right-hand side of (\ref{eq: Pie78}) is always a nonnegative integer. The integer $\mathrm{codim}(Y^\vee)-1$ equals the minimum $i$ such that $\delta_i(Y)\neq 0$. Whenever $Y^\vee$ is a hypersurface, one has
\begin{equation}\label{eq: degdual}
\mathrm{deg}(Y^\vee)=\delta_0(Y)=\sum_{j=0}^m(-1)^{j}(m+1-j)c_j(Y)\cdot h^{m-j}=\sum_{j=0}^m(-1)^{j}(m+1-j)\deg(c_j(Y))\,.
\end{equation}

When $Y$ is not smooth, we can replace Chern classes with {\it Chern-Mather classes}. They are constructed as follows. Let $Y\subset\PP(V)\cong\PP^n$ be a projective variety of dimension $m$. We denote by $G(m+1,V)$ the Grassmannian of $(m+1)$-dimensional vector subspaces of $V$. Consider the {\it Gauss map}
\[
\begin{matrix}
\gamma_Y\colon & Y & \dasharrow & G(m+1,V)\\
& y &  \longmapsto & T_yY
\end{matrix}
\]
which is defined over the smooth points of $Y$. The {\it Nash blow-up} of $Y$ is the closure $\widetilde{Y}$ of the image of $\gamma_Y$. It comes equipped with a proper map $\nu\colon\widetilde{Y}\to Y$. Now let $\mathcal{U}\to G(m+1,V)$ be the {\it universal bundle} over $G(m+1,V)$ of rank $m+1$, where $\mathcal{U}\coloneqq\{(v,W)\in V\times G(m+1,V)\mid v\in W\}$. The vector bundle $\mathcal{U}$ gives a short exact sequence
\[
0\to\mathcal{U}\to\mathcal{O}_{\PP(V)}\otimes V\to\mathcal{Q}\to 0\,,
\]
where $\mathcal{Q}$ denotes the quotient bundle. From this it follows that $\mathcal{Q}\otimes\mathcal{U}^\vee$ is isomorphic to the tangent bundle $\mathcal{T}_{G(m+1,V)}$. The push-forwards under $\nu$ of the Chern classes of the universal bundle restricted to the Nash blow-up $\widetilde{Y}$ are
the {\it Chern-Mather classes} $c_i^\mathsmaller{M}(Y)$ of $Y$. They agree with Chern classes whenever $Y$ is smooth.

The polar classes $\delta_i(Y)$ may be written in terms of the Chern-Mather classes $c_i^\mathsmaller{M}(Y)$, thus generalizing the classical formula in (\ref{eq: Holme}). This generalization is due to Piene (\cite[Theorem 3]{Pie88} and \cite{Pie78}), see also \cite[Proposition 3.13]{Alu}:
\begin{equation}\label{eq: Pie78}
\delta_i(Y)=\sum_{j=0}^{m-i}(-1)^j\binom{m+1-j}{i+1}c_j^\mathsmaller
{M}(Y)\cdot h^{m-j} = \sum_{j=0}^{m-i}(-1)^{j}\binom{m+1-j}{i+1}\deg(c_j^\mathsmaller
{M}(Y))\,.
\end{equation}
In equation (\ref{eq: Pie78}) we use a slightly different convention than in \cite{Alu}. Indeed, for us $c_i^\mathsmaller{M}(Y)$ is the component of dimension $m-i$ (as with standard Chern classes), while in Aluffi's paper it is the component of dimension $i$. We have also the following generalization of equation (\ref{eq: degdual}), when $Y^\vee$  is a hypersurface:
\begin{equation}\label{eq: degdual Mather}
\mathrm{deg}(Y^\vee)=\delta_0(Y)=\sum_{j=0}^m(-1)^j\left(m+1-j\right)c_j^\mathsmaller{M}(Y)\cdot h^{m-j} =\sum_{j=0}^m(-1)^j\left(m+1-j\right)\deg(c_j^\mathsmaller{M}(Y))\,.
\end{equation}
 
\begin{lemma}\label{lem: expressionalphai}
Let $X$ be a projective variety of dimension $m$. For every integer $n\ge 0$, let $Y_n\subset\PP^{n+1}$ be a smooth hypersurface of degree $d$. Then
\[
\delta_0(X\times Y_n)=\sum_{i=0}^m \alpha_i\deg(c_i^\mathsmaller{M}(X))\,, 
\]
where for all $i\in\{0,\ldots,m\}$
\begin{equation}\label{eq: def alphai}
\alpha_i(n,m,d) = \sum_{s=i}^{n+m}(-1)^s(m+n+1-s)\left[\sum_{k=0}^{s-i}\binom{n+2}{k}(-d)^{s-i-k}\right]\binom{m+n-s}{n-s+i}\,.
\end{equation}
\begin{proof}
Write $c^\mathsmaller{M}(X)=\sum_{i=0}^mc_i^\mathsmaller{M}(X)x^i$ and $c^\mathsmaller{M}(Y_n)=c(Y_n)=\frac{(1+y)^{n+2}}{1+d\,y}$ for the Chern-Mather polynomials of $X$ and $Y_n$, respectively. The expression for $c(Y_n)$ is derived from the short exact sequence of sheaves
\[
0\rightarrow \mathcal{T}_{Y_n}\rightarrow \mathcal{T}_{\mathbb P^{n+1}|Y_n}\rightarrow N_{Y_n/\mathbb P^{n+1}}\rightarrow 0\,,
\]
and applying Whitney formula. Keeping into account the relations $x^{m+1}=0=y^{n+1}$, we have
\begin{align*}
c^\mathsmaller{M}(X\times Y_n)&=\left[\sum_{i=0}^mc_i^\mathsmaller{M}(X)x^i\right]\frac{(1+y)^{n+2}}{1+d\,y}\\
&=\left[\sum_{i=0}^mc_i^\mathsmaller{M}(X)x^i\right]\left\{\sum_{j=0}^n\left[\sum_{k=0}^j\binom{n+2}{k}(-d)^{j-k}\right]y^j\right\}\\
&=\sum_{i=0}^m\sum_{j=0}^n\left[\sum_{k=0}^j\binom{n+2}{k}(-d)^{j-k}\right]c_i^\mathsmaller{M}(X)x^iy^j\\
&=\sum_{s=0}^m\left\{\sum_{j=0}^s\left[\sum_{k=0}^j\binom{n+2}{k}(-d)^{j-k}\right]c_{s-j}^\mathsmaller{M}(X)x^{s-j}y^j\right\}\\
&\quad+\sum_{s=m+1}^n\left\{\sum_{j=0}^m\left[\sum_{k=0}^{s-m+j}\binom{n+2}{k}(-d)^{s-m+j-k}\right]c_{m-j}^\mathsmaller{M}(X)x^{m-j}y^{s-m+j}\right\}\\
&\quad+\sum_{s=n+1}^{n+m}\left\{\sum_{j=0}^{n+m-s}\left[\sum_{k=0}^{s-m+j}\binom{n+2}{k}(-d)^{s-m+j-k}\right]c_{m-j}^\mathsmaller{M}(X)x^{m-j}y^{s-m+j}\right\}\\
&=\sum_{s=0}^{n+m}p_s(X)\,,
\end{align*}
where $p_s(X)$ is a homogeneous polynomial of degree $s$ in the variables $x$ and $y$. Using (\ref{eq: degdual Mather}), the polar class $\delta_0(X\times Y_n)$ is given by
\[
\delta_0(X\times Y_n)=\sum_{s=0}^{n+m}(-1)^s(m+n+1-s)p_s(X)\cdot(x+y)^{m+n-s}\,.
\]
A computation reveals that $p_s(X)\cdot(x+y)^{m+n-s}=a_s(X)\,x^my^n$, where
\[
a_s(X) =
\begin{cases}
\sum_{j=0}^s\left[\sum_{k=0}^j\binom{n+2}{k}(-d)^{j-k}\right]\binom{m+n-s}{n-j}c_{s-j}^\mathsmaller{M}(X) & \mbox{for $0\le s\le m$}\,,\\
\sum_{j=0}^m\left[\sum_{k=0}^{s-m+j}\binom{n+2}{k}(-d)^{s-m+j-k}\right]\binom{m+n-s}{m+n-s-j}c_{m-j}^\mathsmaller{M}(X) & \mbox{for $m+1\le s\le n$}\,,\\
\sum_{j=0}^{m+n-s}\left[\sum_{k=0}^{s-m+j}\binom{n+2}{k}(-d)^{s-m+j-k}\right]\binom{m+n-s}{m+n-s-j}c_{m-j}^\mathsmaller{M}(X) & \mbox{for $n+1\le s\le n+m$}\,.
\end{cases}
\]

Plugging in the relations above in the definition of $\delta_0(X\times Y_n)$ and factoring out the Chern-Mather classes $c_i^\mathsmaller{M}(X)$ we derive that
\begin{align*}
\delta_0(X\times Y_n) &=\left\{\sum_{s=0}^n(-1)^s(m+n+1-s)\left[\sum_{k=0}^s\binom{n+2}{k}(-d)^{s-k}\right]\binom{m+n-s}{n-s}\right\}\deg(c_0^\mathsmaller{M}(X))\\
&\quad+\sum_{i=1}^{m-1}\left\{\sum_{s=i}^{n+m}(-1)^s(m+n+1-s)\left[\sum_{k=0}^{s-i}\binom{n+2}{k}(-d)^{s-i-k}\right]\binom{m+n-s}{n-s+i}\right\}\deg(c_i^\mathsmaller{M}(X))\\
&\quad+\left\{\sum_{s=m}^{n+m}(-1)^s(m+n+1-s)\left[\sum_{k=0}^{s-m}\binom{n+2}{k}(-d)^{s-m-k}\right]\right\}\deg(c_m^\mathsmaller{M}(X))\,.
\end{align*}
Therefore, for all $i\in\{0,\ldots,m\}$ the coefficient of $\deg(c_i^\mathsmaller{M}(X))$ is
\[
\alpha_i(n,m,d) = \sum_{s=i}^{n+m}(-1)^s(m+n+1-s)\left[\sum_{k=0}^{s-i}\binom{n+2}{k}(-d)^{s-i-k}\right]\binom{m+n-s}{n-s+i}\,.\qedhere
\]
\end{proof}
\end{lemma}

\begin{theorem}\label{thm: stabilization XxQ}
	Let $X$ be a projective variety of dimension $m$. For every integer $n\ge 0$, let $Y_n\subset\PP^{n+1}$ be a smooth hypersurface of degree $d$. Then
	\[
	\delta_0(X\times Y_n)=(d-1)^{n-m}\delta_0(X\times Y_m)\quad\forall\,n\ge m\,.
	\]
\end{theorem}

\begin{proof}
For all $i\in\{0,\ldots,m\}$, let $\alpha_i=\alpha_i(n,m,d)$ be the coefficient of $\deg(c_i^M(X))$ introduced in \eqref{eq: def alphai}. In what follows, we use that $\binom{a}{b}=0$ if $b$ is a negative integer and we shall sometimes use the formalism of {\it gamma functions} \cite[\S 6]{SS} for convenience.

\noindent In order to prove the statement, it is sufficient to show that
\[
\alpha_i(n+1,m,d)=(d-1)\alpha_i(n,m,d)
\]
for all $n\ge m$ and for all $i\in\{0,\ldots,m\}$. We have that

\begin{align*}
\alpha_i(n+1,m,d)&=\sum_{s=i}^{n+m+1}(-1)^s(m+n+2-s)\left[\sum_{k=0}^{s-i}\binom{n+3}{k}(-d)^{s-i-k}\right]\binom{m+n+1-s}{n+1-s+i} \\
&=(-1)^i(n+m+2-i)\binom{m+n+1-i}{n+1}\\
&\quad+\sum_{s=i+1}^{n+m+1}(-1)^s(m+n+2-s)\left[\sum_{k=0}^{s-i}\binom{n+3}{k}(-d)^{s-i-k}\right]\binom{m+n+1-s}{n+1-s+i}\,.
\end{align*}
For the ease of notation, set $\rho\coloneqq(-1)^i(n+m+2-i)\binom{m+n+1-i}{n+1}$. Then
\begin{align*}
\alpha_i(n+1,m,d)&=\rho+\sum_{r=i}^{n+m}(-1)^{r+1}(m+n+1-r)\left[\sum_{k=0}^{r+1-i}\binom{n+3}{k}(-d)^{r+1-i-k}\right]\binom{m+n-r}{n-r+i}\\
&=\rho+\sum_{r=i}^{n+m}(-1)^{r+1}(m+n+1-r)\left[\sum_{k=0}^{r+1-i}\binom{n+2}{k}(-d)^{r+1-i-k}\right]\binom{m+n-r}{n-r+i}\\
&\quad+\sum_{r=i}^{n+m}(-1)^{r+1}(m+n+1-r)\left[\sum_{k=0}^{r-i}\binom{n+2}{k}(-d)^{r-i-k}\right]\binom{m+n-r}{n-r+i}\\
&=\rho+\sum_{r=i}^{n+m}(-1)^{r+1}(m+n+1-r)\left[\sum_{k=0}^{r-i}\binom{n+2}{k}(-d)^{r+1-i-k}+\binom{n+2}{r+1-i}\right]\binom{m+n-r}{n-r+i}\\
&\quad-\sum_{r=i}^{n+m}(-1)^{r}(m+n+1-r)\left[\sum_{k=0}^{r-i}\binom{n+2}{k}(-d)^{r-i-k}\right]\binom{m+n-r}{n-r+i}\\
&=\rho+d\sum_{r=i}^{n+m}(-1)^{r}(m+n+1-r)\left[\sum_{k=0}^{r-i}\binom{n+2}{k}(-d)^{r-i-k}\right]\binom{m+n-r}{n-r+i}\\
&\quad+\sum_{r=i}^{n+m}(-1)^{r+1}(m+n+1-r)\binom{n+2}{r+1-i}\binom{m+n-r}{n-r+i}\\
&\quad-\sum_{r=i}^{n+m}(-1)^{r}(m+n+1-r)\left[\sum_{k=0}^{r-i}\binom{n+2}{k}(-d)^{r-i-k}\right]\binom{m+n-r}{n-r+i}\\
&=\rho-\sum_{r=i}^{n+m}(-1)^{r}(m+n+1-r)\binom{n+2}{r+1-i}\binom{m+n-r}{n-r+i}\\
&\quad+(d-1)\sum_{r=i}^{n+m}(-1)^{r}(m+n+1-r)\left[\sum_{k=0}^{r-i}\binom{n+2}{k}(-d)^{r-i-k}\right]\binom{m+n-r}{n-r+i}\\
&=(d-1)\alpha_i(n,m,d)+\rho-\sum_{r=i}^{n+m}(-1)^{r}(m+n+1-r)\binom{n+2}{r+1-i}\binom{m+n-r}{n-r+i}\,.
\end{align*}

\noindent To finish off the proof, we have to show the binomial identity
\begin{equation}\label{eq: masterbinomial}
\sum_{r=i}^{n+m}(-1)^{r}(m+n+1-r)\binom{n+2}{r+1-i}\binom{m+n-r}{n-r+i} = (-1)^i(n+m+2-i)\binom{m+n+1-i}{n+1}
\end{equation}
for all $n\ge m$ and for all $i\in\{0,\ldots,m\}$.

\noindent {\bf Case $i=0$ of \eqref{eq: masterbinomial}}. We have to show the identity 
\begin{equation}\label{eq: id1}
\sum_{r=0}^{n}(-1)^{r}(m+n+1-r)\binom{n+2}{r+1}\binom{m+n-r}{n-r}=(n+m+2)\binom{m+n+1}{n+1}\quad\forall n\ge m\,.
\end{equation}
Since $\binom{m+n-r}{n-r}=0$ for $r>n$, we let $r$ run from $0$ to $n$. The summand on the left-hand side of \eqref{eq: id1} is
\begin{align*}
(m+n+1-r)\binom{n+2}{r+1}\binom{m+n-r}{n-r} & = (n+2)\,\frac{(n+1)!}{(n-r+1)!\,(r+1)!}\,\frac{(m+n-r)!}{(n-r)!\,m!}\,(m+n+1-r)\\
& = (n+2)\,\frac{(n+1)!}{(n-r+1)!\,(r+1)!}\,\frac{(m+n+1-r)!}{m!\,(n-r+1)!}\\
& = (n+2)\binom{n+1}{r+1}\binom{m+n+1-r}{m}\,. 
\end{align*}
Therefore, setting $f(n)\coloneqq \sum_{r=0}^{n} (-1)^r \binom{n+1}{r+1}\binom{m+n+1-r}{m}$ and dividing \eqref{eq: id1} by $n+2$, we get that \eqref{eq: id1} is equivalent to 
\begin{equation}\label{eq: id1equiv}
f(n) = \binom{m+n+2}{m}\quad\forall\,n\ge m\,. 
\end{equation}
We prove \eqref{eq: id1equiv} by induction on $n$. Our base case is $n=m$. We invoke  Zeilberger algorithm \cite{Zei} with the \texttt{Maple} \cite{maple} code
\begin{small}
\begin{Verbatim}[commandchars=\\\{\}]
with(SumTools[Hypergeometric]);
T1 := (-1)^r*binomial(m+1,r+1)*binomial(2*m+1-r,m);
ZeilbergerRecurrence(T1,m,r,f,0..m)
\end{Verbatim}
\end{small}
and the output gives the identity
\[
f(m) = \frac{2\Gamma(2+2m)}{(2+m)\Gamma(m+1)^2}\,,
\]
where $\Gamma(z)$ is the {\it gamma function}, whose value on a positive integer $n$ is $\Gamma(n) = (n-1)!$ \cite[\S 6]{SS}. Since $m\in\mathbb N$, we see that 
\[
\frac{2\Gamma(2m+2)}{(m+2)\Gamma(m+1)^2} = \frac{2(2m+1)!}{(m+2)\,m!\,m!}= \frac{2\,(2m+1)!\,(m+1)}{(m+2)\,m!\,m!\,(m+1)} 
=\frac{(2m+2)!}{(m+2)!\,m!} = \binom{2m+2}{m}\,,
\]
which establishes the base case. Suppose now $n>m$. Again, using Zeilberger algorithm, we find the recurrence
\begin{equation}\label{eq: zeilid1}
(1+n-m)f(n) + (n+3) f(n+1) = \frac{2\Gamma(n+3+m)}{\Gamma(n+2)\Gamma(m+1)}\,.
\end{equation}
Using \eqref{eq: zeilid1}, we derive
\begin{align*}
f(n+1) &= \frac{1}{n+3}\,\left[\frac{2\Gamma(n+3+m)}{\Gamma(n+2)\Gamma(m+1)} - (1+n-m)f(n)\right]\\
&= \frac{1}{n+3}\,\left[\frac{2\,(m+n+2)!}{m!\,(n+1)!} - (1+n-m)f(n)\right]\\
&= \frac{1}{n+3}\,\left[2(n+2)\binom{m+n+2}{m} - (1+n-m)f(n)\right]\,.
\end{align*}
By induction hypothesis, one has $f(n) = \binom{m+n+2}{m}$ and hence
\[
f(n+1) = \frac{2(n+2)+m-n-1}{n+3}\binom{m+n+2}{m} = \frac{m+n+3}{n+3}\binom{m+n+2}{m} = \binom{m+n+3}{m}\,,
\]
which verifies the assertion.

\noindent {\bf Case $i=m$ of \eqref{eq: masterbinomial}}. We have to show the identity
\begin{equation}\label{eq: id2}
\sum_{r=m}^{n+m}(-1)^{r}(m+n+1-r)\binom{n+2}{r+1-m} = (-1)^m(n+2)\,. 
\end{equation}
Let $s=r-m$. Thus the left-hand side of equality \eqref{eq: id2} becomes
\[
\begin{gathered}
\sum_{s=0}^n (-1)^{m+s} (n+1-s)\binom{n+2}{s+1} = (-1)^m \sum_{s=0}^n (-1)^s (n+2)\,\frac{(n+1-s)\, (n+1)!}{(s+1)!\,(n-s+1)!} =\\
= (-1)^m (n+2)\,\sum_{s=0}^{n} (-1)^s \frac{(n+1)!}{(s+1)!\,(n-s)!} = (-1)^m (n+2)\,\sum_{s=0}^n (-1)^s \binom{n+1}{s+1} = (-1)^m(n+2)\,,
\end{gathered}
\]
as $\sum_{s=0}^n (-1)^s\binom{n+1}{s+1}=1$. This proves equality \eqref{eq: id2}.

\noindent {\bf Cases $i\in[m-1]$ of \eqref{eq: masterbinomial}}. We have to show the identity
\begin{equation}\label{eq: id3}
\sum_{r=i}^{n+m}(-1)^{r}(m+n+1-r)\binom{n+2}{r+1-i}\binom{m+n-r}{n-r+i} = (-1)^i(n+m+2-i)\binom{m+n+1-i}{n+1}
\end{equation}
for all $i\in[m-1]$ and $n\ge m$. Put $j = m-i$ and $s = r-m+j$. Then \eqref{eq: id3} is equivalent to the equality
\begin{equation}\label{eq: id3equiv}
\sum_{s=0}^{n+j} (-1)^s \frac{(n+1-s+j)!}{\Gamma(n+2-s)\,(s+1)!\,\Gamma(n-s+1)} = \frac{(n+2+j)!}{(n+1)!\,(n+2)!}\,, 
\end{equation}
for all $j\in[m-1]$ and $n\ge m$. Since $j=m-i\geq 1$ and $0\leq s\leq n+j$, the last index we are summing over is $\geq n+1$.
Note that the summands on the left-hand side vanish whenever $s\geq n+1$, therefore we may define
\[ g(n,j)\coloneqq\sum_{s=0}^{n}(-1)^s\frac{(n+1-s+j)!}{\Gamma(n+2-s)\,(s+1)!\,\Gamma(n-s+1)}
\]
and rewrite \eqref{eq: id3equiv} as
\begin{equation}\label{eq: id3equiv2}
g(n,j) = \frac{(n+2+j)!}{(n+1)!\,(n+2)!}\,.
\end{equation}

\noindent Zeilberger algorithm gives the recurrence relation 
\begin{equation}\label{eq: recurid3}
(n+1-j)g(n,j) + (n^2+5n+6)g(n+1,j) = \frac{2\,\Gamma(n+3+j)}{\Gamma(n+2)^2}\,.
\end{equation}
Now, the function 
\[
h(n,j) \coloneqq \frac{(n+2+j)!}{(n+1)!\,(n+2)!}
\]
is seen to satisfy the same recurrence \eqref{eq: recurid3}. One checks $g(j,j) = h(j,j)$ for all $j\in \mathbb N$ by induction. For all $n\geq j$, one has $g(n,j) = h(n,j)$ using \eqref{eq: recurid3} and  $g(j,j) = h(j,j)$ as base case.
\end{proof}

A result due to Weyman and Zelevinsky detects the dual defectiveness of Segre products \cite[Theorem 0.1]{WZ94} assumed in Corollary \ref{cor: stabilization XxQ}. 

\begin{theorem}[{\bf Weyman-Zelevinsky}]\label{thm: WZ94}
	Let $X_1$ and $X_2$ be (embedded) irreducible projective varieties. The dual variety $(X_1\times X_2)^\vee$ is a hypersurface if and only if
	\begin{enumerate}
		\item[(i)] $\mathrm{codim}(X_1^\vee)-1\le \dim(X_2)$, and 
		\item[(ii)] $\mathrm{codim}(X_2^\vee)-1\le \dim(X_1)$.
	\end{enumerate}
\end{theorem}

\noindent We are finally ready to state the main result of this section.

\begin{corollary}\label{cor: stabilization XxQ}
Let $X$ be a projective variety of dimension $m$. For $n\ge 0$, let $Q_n\subset\PP^{n+1}$ be a smooth quadric hypersurface. 
Suppose that $(X\times Q_m)^\vee$ is a hypersurface. Then $(X\times Q_n)^\vee$ is a hypersurface of the same degree as $(X\times Q_m)^\vee$ for all $n\ge m$.
\end{corollary}

\begin{proof}
By Theorem \ref{thm: WZ94}, the variety $(X\times Q_m)^\vee$ is a hypersurface if and only if $m\geq \mathrm{codim}(X^\vee)-1$. If this condition is satisfied, then it is satisfied for all $n\ge m$. In addition, equation \eqref{eq: degdual} gives $\deg[(X\times Q_n)^\vee]=\delta_0(X\times Q_n)$. The statement follows by Theorem \ref{thm: stabilization XxQ} with $d=2$.
\end{proof}

\begin{example}
	Let $X=\PP^1\times\PP^1\subset\PP^3$ and $Q_n$ be a smooth quadric hypersurface in $\PP^{n+1}$. Thus $\dim(X)=2>1=\mathrm{codim}(X^\vee)$. By direct computation, one finds that
	\[
	\deg[(X\times Q_n)^\vee]=
	\begin{cases}
	4 & \mbox{for $n=0$}\\
	12 & \mbox{for $n=1$}\\
	24 & \mbox{for $n=2$}\,.
	\end{cases}
	\]
	Theorem \ref{thm: stabilization XxQ} implies $\deg[(X\times Q_n)^\vee]=24$ for all $n\geq 2=\dim(X)$. 
\end{example}

\subsection{The ED polynomial of a Segre product of two projective varieties}\label{sec: edpolysegre}

The stabilization behavior highlighted in Corollary \ref{cor: stabilization XxQ} has an interesting counterpart related to the ED degree and the ED polynomial of the Segre product between a projective variety and a projective space.
 
For a polynomial function $f$ on a complex vector space $V$, we denote by $\mathcal{V}(f)$ the variety defined by the vanishing of $f$. 
Let $V_1$ and $V_2$ be two complex vector spaces equipped with (real) quadratic forms $q_1$, $q_2$. As in \S\ref{sec: stabilization}, $V_1$ and $V_2$ are complexifications of two real Euclidean spaces. Let $n_i+1 =\dim(V_i)$ and denote $Q_1$ and $Q_2$ the isotropic quadric cones defined by the vanishing of $q_1$ and $q_2$, i.e. $Q_i(x) = \mathcal{V}(q_i(x,x))$.

As above, $V\coloneqq V_1\otimes V_2$ itself is equipped with an Euclidean structure given by the Frobenius inner product $q\coloneqq q_1\otimes q_2$, see Definition \ref{def: Frobenius}. Note that this can be regarded as the familiar space of matrices. Denote by $Q$ the induced isotropic quadric in $V_1\otimes V_2$. 

Consider two affine cones $X_1\subset V_1$ and $X_2\subset V_2$. We also denote by $X_1$ and $X_2$ the corresponding projective varieties in $\PP(V_1)$ and $\PP(V_2)$. Let $X_1\times X_2$ be their Segre product, see Definition \ref{def: segprod}. 

The {\it (squared) distance} $\varepsilon^2=q(t-p)$ between $t\in V_1\otimes V_2$ and a point $p\in(X_1\times X_2)^\vee$ satisfies an algebraic relation of the form
\begin{equation}\label{eq: edpoly}
c_0(t)+c_1(t)\varepsilon^2+\cdots+c_N(t)\varepsilon^{2N} = 0\,,
\end{equation}
where $N=\mathrm{EDdegree}_F\left(X_1\times X_2\right)=\mathrm{EDdegree}_F\left((X_1\times X_2)^\vee\right)$, and the $c_j(t)$'s are homogeneous polynomials in the coordinates of the tensor $t\in V_1\otimes V_2$. 

\begin{definition}[{\bf ED polynomial}]\label{def: EDpoly}
The polynomial on the left-hand side of \eqref{eq: edpoly} is the {\it ED polynomial} of $(X_1\times X_2)^\vee$. This is a univariate polynomial 
in $\varepsilon^2$ and is denoted $\mathrm{EDpoly}_{(X_1\times X_2)^\vee,t}(\varepsilon^2)$. 
\end{definition}

A consequence of \cite[Corollary 5.5]{OS} is the following result.

\begin{proposition}\label{prop: vanishing locus lowest coefficient}
Assume that $(X_1\times X_2)\cap Q$ is a reduced variety. Then the locus of all $t\in V_1\otimes V_2$ where $\mathrm{EDpoly}_{(X_1\times X_2)^\vee,t}(0)$ vanishes is
\[
(X_1\times X_2)^\vee \cup [(X_1\times X_2)\cap Q]^\vee = (X_1\times X_2)^\vee \cup [(X_1\cap Q_1)\times X_2]^\vee \cup [X_1\times (X_2\cap Q_2)]^\vee \,.
\]
Furthermore, by \cite[Proposition 4.4.13]{Sod20}, whenever $(X_1\times X_2)^\vee$ is a hypersurface, then its defining polynomial appears with multiplicity two in the term $c_0$.
\end{proposition}

A similar argument used in the proof of \cite[Proposition 5.2.6]{Sod20} leads to the following inclusion.

\begin{proposition}\label{prop: inclusion leading coefficient}
The following inclusion holds true:
\[
\mathcal{V}(c_N) \subset [(X_1\cap Q_1)\times (X_2\cap Q_2)]^\vee \,.
\]
\end{proposition}
In other words, if $t\in V_1\otimes V_2$ admits strictly less critical points than $N$, then it is forced to have a specific isotropic structure.

Summing up, we may write the extreme coefficients of $\mathrm{EDpoly}_{(X_1\times X_2)^\vee,t}(\varepsilon^2)$ as
\[
c_0 = f^2\,g^\alpha,\quad c_N = h^\beta\,,
\]
for some square-free polynomials $f$, $g$, $h$, where
\begin{align*}
	\mathcal{V}(f) &= (X_1\times X_2)^\vee,\\
	\mathcal{V}(g) &= [(X_1\cap Q_1)\times X_2]^\vee \cup [X_1\times (X_2\cap Q_2)]^\vee,\\
	\mathcal{V}(h^\beta) &= [(X_1\cap Q_1)\times (X_2\cap Q_2)]^\vee\,,
\end{align*}
whenever the varieties on the right-hand side are hypersurfaces. (The polynomials are set to be $1$ if the corresponding varieties have higher codimensions.) 

Note that when $X_1=V_1$ and $X_2=V_2$ we are looking at the distance function from a Segre product of projective spaces. An immediate consequence of the {\it Eckart-Young Theorem} tells us that (assuming $n_1\le n_2$)
\[
\mathrm{EDpoly}_{(\PP(V_1)\times \PP(V_2))^\vee,t}(\varepsilon^2) = \det(t\,t^\mathsmaller{T}-\varepsilon^2I_{n_1})\,.
\]
In particular, $\mathrm{EDpoly}_{(\PP(V_1)\times \PP(V_2))^\vee,t}(\varepsilon^2)$ is a monic polynomial, i.e. the exponent $\beta$ of $h$ is zero. On the other hand, the lowest coefficient is $\det(t\,t^\mathsmaller{T})$. When $n_1=n_2$, then $\det(t\,t^\mathsmaller{T})=\det(t)^2=f^2$, whereas $g=1$ because its corresponding variety is not a hypersurface. Otherwise $n_1<n_2$ and then $\det(t\,t^\mathsmaller{T})=g$, whereas in this case $f=1$ because $(\PP(V_1)\times \PP(V_2))^\vee$ is not a hypersurface. That means that the exponent of $g$ is $\alpha=1$.

These observations lead to the following more general conjecture, confirmed by experimental data from the software \texttt{Macaulay2}. 

\begin{conjecture}\label{conj: exponents}
Assume that $(X_1\times X_2)\cap Q$ is a reduced variety. Then the extreme coefficients $c_0$ and $c_N$ of $\mathrm{EDpoly}_{(X_1\times X_2)^\vee,t}(\varepsilon^2)$ are respectively
\[
c_0 = f^2\,g,\quad c_N\in\R\,,
\]
where $\mathcal{V}(f) = (X_1\times X_2)^\vee$ and $\mathcal{V}(g) = [(X_1\cap Q_1)\times X_2]^\vee \cup [X_1\times (X_2\cap Q_2)]^\vee$.
\end{conjecture}

The validity of Conjecture \ref{conj: exponents} implies the stabilization of the ED degree of $X\times\PP(V_2)$ for $n_2$ increasing.

\begin{proposition}
Assume Conjecture \ref{conj: exponents} is true. Let $X\subset \PP(V_1)$ be a projective variety and assume that $X\cap Q_1$ is reduced.
Then $\mathrm{EDdegree}_F(X\times\PP(V_2))$ stabilizes for $n_2\ge\dim(X)+1$. 
\begin{proof}
By Definition \ref{def: EDpoly}, we have the equality
\[
\deg(c_N)+2\,\mathrm{EDdegree}_F(X\times\PP(V_2)) = \deg(c_0)\,,
\]
where $\mathcal{V}(c_0)$ and $\mathcal{V}(c_N)$ are described respectively in Proposition \ref{prop: vanishing locus lowest coefficient}  and Proposition \ref{prop: inclusion leading coefficient}. 
		
By Theorem \ref{thm: WZ94}, the varieties $(X\times\PP(V_2))^\vee$ and $[(X\cap Q_1)\times \PP(V_2)]^\vee$ are not hypersurfaces for all $n_2\ge\dim(X)+1$, whereas the variety $(X\times Q_2)^\vee$ is a hypersurface for all $n_2\ge\dim(X)+1$. Therefore, for all $n_2\ge\dim(X)+1$ we have $c_0=g^\alpha$ for some positive integer $\alpha$, where $\mathcal{V}(g)=(X_1\times Q_2)^\vee$.

By the assumption on the Conjecture \ref{conj: exponents}, we conclude that $\alpha=1$ and $\deg(c_N)=0$. In particular, we derive
\[
2\,\mathrm{EDdegree}_F(X\times\PP(V_2)) = \deg((X\times Q_2)^\vee)\quad\forall n_2\ge\dim(X)+1\,.
\]
Conclusion follows from Corollary \ref{cor: stabilization XxQ}. 
\end{proof}
\end{proposition}

The previous result proves a stabilization property of the ED degree of the Segre product $X\times\PP(V_2)$. Furthermore, some experiments with the software  \texttt{Macaulay2} suggest the following conjecture which is a somewhat more general version of Conjecture \ref{conj: stabthm}.

\begin{conjecture}
Let $X\subset\PP(V_1)$ be a projective hypersurface such that $X\cap Q_1$ is reduced. Consider the Segre product $X\times\PP(V_2)$. Then
\[
\mathrm{EDdegree}_F(X\times\PP(V_2))=\sum_{j=0}^{n_1-1}\mathrm{EDdegree}_F(X\cap L_j)\,,
\]
where $L_j\subset\PP(V_1)$ is a general subspace of codimension $j$.
\end{conjecture}

\section*{Acknowledgements}

We thank Jay Pantone for very useful conversations. We thank the referees for useful comments. The first two authors are members of INDAM-GNSAGA. The third author would like to thank the Department of Mathematics of Universit\`{a} di Firenze, where this project started in June 2018, for the warm hospitality and financial support.
The first author is supported by the H2020-MSCA-ITN-2018 project POEMA.
The second author is partially supported by the Academy of Finland Grant 323416.
The third author is supported by Vici Grant 639.033.514 of Jan Draisma from the Netherlands Organisation for Scientific Research.

\end{document}